\documentclass[10pt]{article}
\usepackage[english]{babel}
\usepackage{latexsym,amsfonts,amsmath,enumerate,graphics,enumerate,amsthm,amssymb}
\title{\sc {\huge Isometries of polyhedral {H}ilbert geometries}}
\author{\sc Bas Lemmens\\
\footnotesize Mathematics Institute, University of Warwick\\
\footnotesize CV4 7AL Coventry, United Kingdom\\
\footnotesize E-mail: {\tt B.Lemmens@warwick.ac.uk}\\
\mbox{}\\
\sc Cormac Walsh\footnote{C. Walsh was partially supported by the joint
RFBR-CNRS grant number 05-01-02807 }\\
\footnotesize INRIA Saclay \& CMAP, \'{E}cole Polytechnique,\\ 
\footnotesize 91128 Palaiseau, France\\
\footnotesize E-mail: {\tt cormac.walsh@inria.fr}}

\date{\today}

\let\epsilon=\varepsilon
\let\phi=\varphi
\let\theta=\vartheta
\newtheorem{theorem}{Theorem}[section]
\newtheorem{Definition}[theorem]{Definition}
\newtheorem{lemma}[theorem]{Lemma}
\newtheorem{corollary}[theorem]{Corollary}
\newtheorem{proposition}[theorem]{Proposition}

\begin{document}
\maketitle

{\sc Abstract.-} We show that the isometry group of a polyhedral Hilbert
geometry coincides with its group of collineations (projectivities)
if and only if the polyhedron is not an $n$-simplex with $n\geq 2$.
Moreover, we determine the isometry group of the Hilbert geometry on the
$n$-simplex for all $n\geq 2$, and find that it has the collineation group
as an index-two subgroup. These results confirm, for the class of polyhedral
Hilbert geometries, several conjectures posed by P.~de~la~Harpe. \\
\mbox{}

{\sc AMS Classification (2000):}
53C60,  
22F50 \\ 
\mbox{}

\mbox{}

\section{Introduction}\label{sec:1}

In a letter to Klein, Hilbert remarked that every open bounded convex subset
$X$ of $\mathbb{R}^n$ can be equipped with a metric
$d_X\colon X\times X\to[0,\infty)$ defined by 
\begin{equation}\label{eq:1.1}
d_X(x,y)=\log\,[x',x,y,y'], 
\end{equation}
where $x',y'\in\partial X$, the points $x',x,y,y'$ are aligned in this order,
and 
\begin{equation}\label{eq:1.2} 
[x',x,y,y']=\frac{|x'y|\,|y'x|}{|x'x|\,|y'y|}
\end{equation}
is the \emph{cross-ratio}. This metric is called the \emph{Hilbert metric}
and $(X,d_X)$ is said to be the \emph{Hilbert geometry} on $X$.

As Hilbert noted \cite{Hil}, if $X$ is an open $n$-dimensional ellipsoid,
then $(X,d_X)$ is a model for the hyperbolic $n$-space. On the other hand,
if $X$ is an open $n$-simplex, then $(X,d_X)$ is isometric to a normed space.
More precisely, let $V=\mathbb{R}^{n+1}/\sim$, where $x\sim y$ if
$x=y +\alpha (1,1,\ldots,1)$ for some $\alpha \in\mathbb{R}$, and equip the
vector space $V$ with the \emph{variation norm}: 
\begin{equation*}
\|x\|_{\mathrm{var}}= \max_i x_i - \min_j x_j. 
\end{equation*}
It is known \cite[Proposition 1.7]{Nu1} that if $X$ is an $n$-dimensional
simplex, then $(X,d_X)$ is isometric to $(V,\|\cdot\|_{\mathrm{var}})$. 

Hilbert geometries display features of negative curvature and are of interest
in metric geometry. The extent to which the shape of the domain $X$ affects
the geometry of $(X,d_X)$ has been the subject of numerous studies,
for example~\cite{Be,Bu,CV,FK,KN,So1,So2,Wa1}. The Hilbert metric also has
striking applications in the spectral theory of (non-linear) operators on
cones in a Banach space; see, for instance, \cite{Bi,Bus,LN,Nu1,Sa}.

We study the group of isometries $\mathrm{Isom}(X)$ of the Hilbert geometry
when the domain $X$ is a polyhedron in $\mathbb{R}^n$, in other words, when $X$
is the intersection of finitely many open half-spaces. For simplicity, we call
such Hilbert geometries \emph{polyhedral}.

Natural isometries arise from collineations (projectivities) of $X$. Indeed,
let $\mathbb{P}^n=\mathbb{R}^n\cup\mathbb{P}^{n-1}$ be the real
$n$-dimensional projective space. Suppose that $X$ is contained in the
open cell $\mathbb{R}^n$ inside $\mathbb{P}^n$, and let
$\mathrm{Coll}(X)=\{h\in\mathrm{PGL}(n,\mathbb{R})\colon h(X)=X\}$ be the group
of collineations that map $X$ onto itself. As every collineation preserves the
cross-ratio, we have that $\mathrm{Coll}(X)\subseteq \mathrm{Isom}(X)$.

In~\cite{dlH}, de la Harpe raised a number of questions concerning
$\mathrm{Isom}(X)$ and its relation to $\mathrm{Coll}(X)$. In particular,
he conjectured that $\mathrm{Isom}(X)$ is a Lie group, and that
$\mathrm{Isom}(X)$ acts transitively on $X$ if and only if $\mathrm{Coll}(X)$
does. He also asked for which sets $X$ the groups 
$\mathrm{Isom}(X)$ and $\mathrm{Coll}(X)$ coincide.
Of course, if the two groups are equal, then $\mathrm{Isom}(X)$ is a Lie group,
since $\mathrm{Coll}(X)$ is a closed subgroup of $\mathrm{PGL}(n,\mathbb{R})$.
De la Harpe~\cite[Proposition 3]{dlH} proved that if the norm closure
$\overline{X}$ of $X$ is strictly convex,
then the groups are equal. He also determined $\mathrm{Isom}(X)$ when $X$ is an
open $2$-simplex and showed that $\mathrm{Isom}(X)=\mathrm{Coll}(X)$ when $X$
is an open quadrilateral in the plane.

Our main results are the following two theorems, which confirm de la Harpe's
conjectures for the class of polyhedral Hilbert geometries. 
\begin{theorem}\label{thm:1.1} 
If $(X,d_X)$ is a polyhedral Hilbert geometry, then 
\[
\mathrm{Isom}(X)=\mathrm{Coll}(X)
\]
if and only if $X$ is not an open $n$-simplex with $n\geq 2$. 
\end{theorem} 
We also determine the isometry group in the case of the $n$-simplex.
Let $\sigma_{n+1}$ be the group of coordinate
permutations on $V$, let $\rho\colon V\to V$ be the isometry given by
$\rho(x)=-x$ for $x\in V$, and identify the group of translations in $V$
with $\mathbb{R}^n$.
\begin{theorem}\label{thm:1.2} 
If $X$ is an open $n$-simplex with $n\geq 2$, then 
\[
\mathrm{Coll}(X) \cong \mathbb{R}^n\rtimes \sigma_{n+1}
\quad\mbox{and}\quad 
\mathrm{Isom}(X) \cong \mathbb{R}^n\rtimes \Gamma_{n+1}, 
\]
where $\Gamma_{n+1} =  \sigma_{n+1}\times\langle\rho\rangle$.
\end{theorem} 

It is clear from this that the collineation group of the $n$-simplex
($n\geq 2$) is a subgroup of index two in the isometry group.

\section{Birkhoff's version of the Hilbert metric}\label{sec:2} 

In \cite{Bi} Birkhoff used the Hilbert metric to analyse the spectral
properties of linear operators that leave a closed cone in a Banach space
invariant, which  led him to consider another version of the Hilbert metric.
We shall use both versions in this paper. In Birkhoff's setting, one considers
an open cone $C\subseteq\mathbb{R}^{n+1}$, that is, $C$ is open and convex
and $\lambda C\subseteq C$ for all $\lambda > 0$.
If, in addition, $\overline{C}\cap (-\overline{C})=\{0\}$, then we call $C$
a \emph{proper} open cone. An open cone $C$ induces a pre-order $\leq_C$ on
$\mathbb{R}^{n+1}$ by $x\leq_C y$ if $y-x\in \overline{C}$.
If $C$ is a proper open cone, then $\leq_C$ is also anti-symmetric and hence 
a partial ordering on $\mathbb{R}^{n+1}$. 

For $x\in C$ and $y\in\mathbb{R}^{n+1}$, define 
\begin{equation*}
M(y/x;C) =\inf\{\lambda > 0\colon y\leq_C\lambda x\}.
\end{equation*}
Note that $M(y/x;C)$ is finite since $C$ is open.
Also note that, by the Hahn-Banach separation theorem, $x\leq_C y$
if and only if $\langle\phi,x\rangle\leq \langle\phi,y\rangle$
for all $\phi\in C^*$,
where $C^*=\{\phi\in\mathbb{R}^{n+1}\colon \mbox{$\langle\phi,x\rangle\geq 0$
for all $x\in \overline{C}$}\}$
is the \emph{dual cone} of $\overline{C}$. Thus, 
\begin{equation}\label{eq:2.2} 
M(y/x;C)=\sup_{\phi\in C^*\setminus\{0\}}\frac{\langle\phi,y\rangle}{\langle\phi,x\rangle}
\mbox{\quad for all $x\in C$ and $y\in\mathbb{R}^{n+1}$.}
\end{equation}

Birkhoff's version of the Hilbert metric is called
\emph{Hilbert's projective metric} on $C$ and is defined by 
\begin{equation*}
d_C(x,y)=\log M(x/y;C) + \log M(y/x;C)\mbox{\quad for all }x,y\in C.
\end{equation*}
Note that $d_C(\alpha x,\beta y)=d_C(x,y)$ for all $\alpha,\beta>0$. 
It is known~\cite{Nu1} that $d_C$ is a semi-metric on the rays in $C$,
but in general not a metric, as $d_C(x,y) =0 $ does not imply $x=\alpha y$
for some $\alpha>0$. If, however, $C$ is a proper open cone,
then $d_C$ is a genuine metric on the rays in $C$.
To establish the connection with the Hilbert metric, we imagine $X$ as a subset
of a hyperplane in $\mathbb{R}^{n+1}$ that does not contain the origin.
Let $C_X$ be the cone generated by $X$ in $\mathbb{R}^{n+1}$. So, 
\[
C_X=\{\lambda x\in\mathbb{R}^{n+1}\colon \lambda> 0\mbox{ and }x\in X\}
\]
is a proper open cone in $\mathbb{R}^{n+1}$.
Birkhoff \cite{Bi} proved that $d_C$ and $d_X$ coincide on $X$. In fact, 
\begin{equation*}
\log M(x/y;C)=\log\frac{|y'x|}{|y'y|}
\mbox{\quad and \quad} \log M(y/x;C)=\log\frac{|x'y|}{|x'x|}
\end{equation*}
for all $x,y\in X$. 
We write 
\[
\mathcal{F}_C(x,y) = \log M(x/y;C) 
\]
for all $y\in C$ and $x\in\mathbb{R}^{n+1}$, and 
\[
\mathcal{RF}_C(x,y) = \log M(y/x;C)
\]
for all $x\in C$ and $y\in\mathbb{R}^{n+1}$. 

The function $\mathcal{F}_C$ is called the \emph{Funk metric} after P.~Funk
who used it in~\cite{Funk}.
It is easy to verify that
$\mathcal{F}_C(x,z)\leq \mathcal{F}_C(x,y)+\mathcal{F}_C(y,z)$
and $\mathcal{F}_C(x,x) = 0$ for all $x,y,z\in C$, but $\mathcal{F}_C(x,y)$
is neither symmetric nor non-negative.
We call $\mathcal{RF}_C$ the \emph{reverse-Funk metric}.

We have that
\begin{equation*}
d_C(x,y) = \mathcal{F}_C(x,y)+\mathcal{RF}_C(x,y)
\mbox{\quad for all $x,y\in C$}.
\end{equation*}
We write $[0]_C$ to denote the subspace
$\{x\in \mathbb{R}^{n+1}\colon
   \mbox{$x\in \overline{C}$ and $-x\in \overline{C}$}\}$.
Clearly if $z\in [0]_C$, then $\langle\phi,z\rangle=0$ for all $\phi\in C^*$.
From (\ref{eq:2.2}) we deduce that if $z\in [0]_C$ and $x,y\in C$, then 
\[
M((x+z)/y;C)=M(x/y;C)\mbox{\quad and \quad } M(y/(x+z);C) = M(y/x;C), 
\]
so that $d_C(x+z,y)=d_C(x,y)$.
Therefore, if $\Sigma$ is a cross-section of the proper open cone $C'=C/[0]_C$,
then $(\Sigma,d_{C'})$ is isometric to a Hilbert geometry with dimension
$n-\dim [0]_C$.
We call $n-\dim [0]_C$ the \emph{dimension of the Hilbert geometry on }$C$. 
 
\section{The horoboundary and the detour metric}\label{sec:3}
To prove Theorem \ref{thm:1.1}, we use results from \cite{Wa1} on the
horofunction boundary of the Hilbert geometry.
Following \cite{BGS}, recall that if $(X,d)$ is an unbounded locally-compact
metric space, then to each $z\in X$ a continuous function
$\phi_{z,b}\colon X\to\mathbb{R}$, with 
\[
\phi_{z,b}(x)=d(x,z)-d(b,z)\mbox{\quad for }x\in X, 
\]
is assigned. Here $b\in X$ is a fixed \emph{base-point}.
The map $\Phi\colon X\to C(X)$ given by $\Phi(z)=\phi_{z,b}$ embeds $X$
into the space of continuous functions on $X$, which is endowed with the
topology of uniform convergence on compact subsets of $X$.
The \emph{horoboundary} of $X$ is defined by
\[
X(\infty)=\overline{\Phi(X)}\setminus\Phi(X),
\] 
and its members are called \emph{horofunctions}.
Since $X$ is locally compact, the space $X\cup X(\infty)$ is a compactification
of $X$, and so every unbounded sequence $(z_k)_k$ in $X$
has a subsequence such that $\phi_{z_k,b}$ converges to a point in $X(\infty)$.

It is easy to verify that, for any alternative base-point $b'$,
\[
\phi_{z_k,b'}(x) = \phi_{z_k,b}(x)-\phi_{z_k,b}(b').
\]
Therefore, if $\phi_{z_k,b}$ converges to $\xi$, then $\phi_{z_k,b'}$
converges to $\xi-\xi(b')$.

If $r\colon [0,\infty)\to X$ is a geodesic ray, then $\phi_{r(t),r(0)}(x)$
is non-increasing and bounded below by $-d(r(0),x)$. Therefore, each geodesic
ray yields a horofunction. More generally, one obtains a horofunction from
each ``almost-geodesic'', a concept introduced by Rieffel~\cite{Rief}.
A map $\gamma\colon T\to X$, with $T$ an unbounded subset of $\mathbb{R}$
and $0\in T$, is called an \emph{almost-geodesic} if for each $\epsilon>0$
there exists $M\geq 0$ such that
\begin{equation}\label{eq:3.1} 
|d(\gamma(t),\gamma(s))+d(\gamma(s),\gamma(0))-t|<\epsilon
\mbox{\quad for all $s,t\in T$ with $t\geq s\geq M$}.
\end{equation}
Rieffel~\cite{Rief} proved that, for any almost-geodesic $\gamma\colon T\to X$,
the quantity
$d(x,\gamma(t))-d(b,\gamma(t))$ converges to some limit $\xi(x)$
for each $x\in X$. In this case, we say that $\gamma$ \emph{converges to} $\xi$.
A horofunction $\xi\in X(\infty)$ is called a \emph{Busemann point} if there
exists an almost-geodesic converging to it.
We denote by $X_B(\infty)$ the set of all Busemann points in $X(\infty)$.

It was shown in \cite{AGW} that the Busemann points can also be obtained
as limits of $\epsilon$-almost-geodesics. Recall that a sequence $(x_k)_k$
in $X$ is called an \emph{ $\epsilon$-almost-geodesic} if 
\[
d(x_0,x_1)+\dots+d(x_m,x_{m+1})\leq d(x_0,x_{m+1})+\epsilon
\mbox{\quad for all $m\geq 0$}.
\]
In fact, it was shown in~\cite[Proposition 7.12]{AGW} that every
almost-geodesic has a subsequence that is an $\epsilon$-almost-geodesic
for some $\epsilon>0$, and, conversely, every unbounded
$\epsilon$-almost-geodesic has a subsequence that is an almost-geodesic.  

For any two Busemann points $\xi$ and $\eta$, we define the \emph{detour cost}
by 
\begin{align*}
H(\xi,\eta) &=\sup_{W\ni\xi} \inf_{x\in W} d(b,x)+\eta(x), 
\end{align*}
where the supremum is taken over all neighbourhoods $W$ of $\xi$ in the
compactification $X\cup X(\infty)$.
This concept originated in~\cite{AGW}.
An equivalent definition is
\begin{align}\label{eq:3.2}
H(\xi,\eta) &= \inf_{\gamma} \liminf_{t\to\infty}
                    d(b,\gamma(t))+\eta(\gamma(t)), 
\end{align}
where the infimum is taken over all paths $\gamma:T\to X$ converging to $\xi$.

The following is a special case of~\cite[Lemma 3.3]{Wa0}.
\begin{lemma}\label{lem:3.2} 
Let $\gamma$ be an almost-geodesic converging to a Busemann point $\xi$.
Then,
\begin{equation*}
\lim_{t\to\infty} d(b,\gamma(t)) + \xi(\gamma(t)) = 0.
\end{equation*}
Moreover, for any horofunction $\eta$,
\begin{equation*}
\lim_{t\to\infty} d(b,\gamma(t)) + \eta(\gamma(t)) = H(\xi,\eta).
\end{equation*}
\end{lemma}
\begin{proof}
Let $\epsilon>0$ and assume that $b=\gamma(0)$.
As $\gamma$ is an almost-geodesic we have that 
\[
d(\gamma(0),\gamma(t))
   \geq d(\gamma(0),\gamma(s)) + d(\gamma(s),\gamma(t)) - \epsilon
\]
for all $s$ and $t$ sufficiently large, with $s\le t$.
Subtracting $d(\gamma(0),\gamma(t))$ from both sides and letting
$t$ tend to infinity gives
\[
0\geq d(\gamma(0),\gamma(s))+\xi(\gamma(s))-\epsilon
\mbox{\quad for all $s$ sufficiently large.}
\]
This implies that 
\[
\limsup_{s\to\infty} d(\gamma(0),\gamma(s))+\xi(\gamma(s))\leq 0.
\]
As $d(\gamma(0),\gamma(s))+d(\gamma(s),\gamma(t))-d(\gamma(0),\gamma(t))\geq 0$ for all $t$, we see that 
\[
\liminf_{s\to\infty} d(\gamma(0),\gamma(s))+\xi(\gamma(s))\geq 0, 
\]
which proves the first statement when $b=\gamma(0)$.
The equality for general $b$ follows from the fact that if $\gamma$ converges
to $\xi$ with respect to the base-point $\gamma(0)$, then $\gamma$ converges
to $\xi'=\xi-\xi(b)$ with respect to the base-point $b$. 

Observe that
\begin{align*}
\eta(x) \le \Big( d(x,z) - d(b,z) \Big) + \Big( d(b,z) + \eta(z) \Big)
\quad \text{for all $x$ and $z$ in $X$}.
\end{align*}
It follows that
\begin{align*}
\eta(x) \le \xi(x) + H(\xi,\eta)
\quad \text{for all $x$ in $X$}.
\end{align*}
So,
\begin{align*}
d(b,\gamma(t)) + \eta(\gamma(t))
   \le d(b,\gamma(t)) + \xi(\gamma(t)) + H(\xi,\eta)
\quad \text{for all $t$}.
\end{align*}
Taking the limit supremum as $t$ tends to infinity and using the first part
of the lemma, we get that
\begin{align*}
\limsup_{t\to\infty} d(b,\gamma(t)) + \eta(\gamma(t)) \le H(\xi,\eta).
\end{align*}
The lower bound on the limit infimum follows from~(\ref{eq:3.2}):
\begin{equation*}
H(\xi,\eta)
   \le \liminf_{t\to\infty} d(b,\gamma(t)) + \eta(\gamma(t)) .
\end{equation*}
Thus, the second statement is proved.
\end{proof}
In particular, we see that $\lim_{k\to\infty} d(b,\gamma(t))+\eta(\gamma(t))$ is independent of the almost geodesic $\gamma$ converging to $\xi$. 

By symmetrising the detour cost, the set of Busemann points can be equipped
with a metric. For $\xi$ and $\eta$ in $X_B(\infty)$, we define 
\begin{equation}\label{eq:3.4}
\delta(\xi,\eta) = H(\xi,\eta)+H(\eta,\xi) 
\end{equation}
and call $\delta$ the \emph{detour metric}. 
This construction appears in~\cite[Remark~5.2]{AGW}.
\begin{proposition}\label{prop:3.3} 
The function $\delta\colon X_B(\infty)\times X_B(\infty)\to [0,\infty]$
is a metric, which might take the value $+\infty$.  
\end{proposition}
\begin{proof}
Clearly $\delta$ is symmetric.

Let $\gamma$ and $\lambda$ be almost-geodesics converging, respectively,
to $\xi$ and $\eta$. From the triangle inequality we get that 
\[
d(b,\gamma(t))+d(\gamma(t),\lambda(s))-d(b,\lambda(s))\geq 0. 
\]
Letting $s$ tend to infinity, we find that
\begin{equation}\label{eq:3.5}
d(b,\gamma(t))+\eta(\gamma(t))\geq 0,
\end{equation}
so that $H(\xi,\eta)\geq 0$. We conclude that $\delta$
is non-negative.

From Lemma \ref{lem:3.2}, it follows that $\delta(\xi,\xi)= 0$
for all $\xi\in X_B(\infty)$.

Now suppose that $\delta(\xi,\eta)=0$. To show that $\xi=\eta$, we let $x\in X$.
By~(\ref{eq:3.5}) we know that, for all $s$,
\begin{eqnarray*}
d(x,\gamma(t))-d(b,\gamma(t))
  &\leq& d(x,\lambda(s))+d(\lambda(s),\gamma(t))+\eta(\gamma(t))\\
  &=& d(x,\lambda(s))+\big{(}d(\lambda(s),\gamma(t))- d(b,\gamma(t))\big{)} \\
  & & \quad\quad  + \big{(}d(b,\gamma(t))+\eta(\gamma(t))\big{)}. 
\end{eqnarray*}
Taking the limit as $t$ tends to infinity gives, by Lemma~\ref{lem:3.2},
\begin{eqnarray*}
\xi(x)
   &\leq&  d(x,\lambda(s))+\xi(\lambda(s))+ H(\xi,\eta)\\
   &=&  \big{(}d(x,\lambda(s))-d(b,\lambda(s))\big{)}
       +\big{(} d(b,\lambda(s)) + \xi(\lambda(s))\big{)} + H(\xi,\eta). 
\end{eqnarray*}
Subsequently letting $s$ tend to infinity shows that 
\[
\xi(x) \leq   \eta(x)+ H(\eta,\xi) + H(\xi,\eta) = \eta(x).
\]
Interchanging the roles of $\xi$ and $\eta$ gives the desired equality. 

It remains to show that $\delta$ satisfies the triangle inequality. 
Let $\xi$, $\eta$, and $\nu$ be Busemann points with respective
almost-geodesics $\gamma$, $\lambda$, and $\kappa$. Clearly 
\begin{align*}
d(b,\gamma(t)) +d(\gamma(t),\kappa(u)) - d(b,\kappa(u))  & \leq 
  d(b,\gamma(t)) + d(\gamma(t),\lambda(s)) -d(b,\lambda(s)) \\ 
    & \quad + d(b,\lambda(s)) +d(\lambda(s),\kappa(u)) - d(b,\kappa(u)).
  \end{align*}
Taking the limits as $u$, $s$, and then $t$ tend to infinity, we get that
$H(\xi,\nu)\leq H(\xi,\eta)+H(\eta,\nu)$, which implies that $\delta$ satisfies
the triangle inequality.  
\end{proof} 

Note that we can partition $X_B(\infty)$ into disjoint subsets such that
$\delta(\xi,\eta)$ is finite for each pair of horofunctions $\xi$ and $\eta$
lying in the same subset. We call these subsets the \emph{parts} of the
horofunction boundary of $(X,d)$, and $\delta$ is a genuine metric on each
one.

Consider an isometry $g$ from one metric space $(X,d)$ to another $(Y,d')$.
We can extend $g$ to the horofunction boundary $X(\infty)$ of $X$ as follows:
\[
g(\xi)(y) = \xi(g^{-1}(y))-\xi(g^{-1}(b')),
\]
for all $\xi\in X(\infty)$ and $y\in Y$. Here $b'$ is the base-point in $Y$.
Observe that if $\lambda\colon T\to X$ is a path converging to a horofunction
$\xi$, then $g\circ\lambda$ converges to $g(\xi)$ in the horofunction
compactification $Y\cup Y(\infty)$ of $Y$.
If, furthermore, $\lambda$ is an almost-geodesic, then $g\circ \lambda$
is an almost-geodesic in $(Y,d')$.

The following lemma shows that $g$ is an isometry on $X_B(\infty)$
with respect to the detour metric. 
The first part has appeared in~\cite[Remark~5.2]{AGW}.
\begin{lemma}\label{lem:3.4} 
The detour metric $\delta$ is independent of the base-point.
Moreover, if $g\colon (X,d)\to (Y,d')$ is an isometry of $X$ onto $Y$, then
\[
\delta(\xi,\eta)=\delta(g(\xi),g(\eta))
\mbox{\quad for all $\xi,\eta\in X(\infty)$}.
\]
\end{lemma}
\begin{proof}
Let $\xi$ and $\eta$ be horofunctions with respect to the base-point $b\in X$.
Now let $\hat b\in X$ and note that $\hat\xi = \xi -\xi(\hat b)$ and
$\hat\eta=\eta-\eta(\hat b)$
are the corresponding horofunctions when using $\hat b$ as the base-point
instead of $b$.
So,
\begin{align*}
H(\hat\xi,\hat\eta)
   &=\inf_\gamma \liminf_{t\to\infty} d(\hat b,\gamma(t)) + \hat\eta(\gamma(t)) \\
   &=\inf_\gamma \liminf_{t\to\infty}
                    d(\hat b,\gamma(t)) - d(b,\gamma(t)) +d(b,\gamma(t))+ \eta(\gamma(t))- \eta(\hat b)\\
   &= \xi(\hat b) + H(\xi,\eta) -\eta(\hat b),
\end{align*}   
where each time the infimum is taken over all paths converging to $\xi$.
This implies that $\delta(\hat\xi,\hat\eta)=\delta(\xi,\eta)$. 

Let $b'$ be the base-point of $Y$. We have
\begin{align*}
H(g(\xi),g(\eta))
   &= \inf_\gamma \liminf_{t\to\infty}
         d'(b',g(\gamma(t)))+\eta(\gamma(t))-\eta(g^{-1}(b')) \\
   &= \inf_\gamma \liminf_{t\to\infty}
         d(g^{-1}(b'),\gamma(t))-d(b,\gamma(t)) +d(b,\gamma(t)) + \eta(\gamma(t)) -
         \eta( g^{-1}(b'))\\
                  & = \xi(g^{-1}(b')) + H(\xi,\eta) - \eta( g^{-1}(b')),
\end{align*}   
where the infimum is as before.
We conclude that $g$ preserves the detour cost.
\end{proof}

\section{Parts of the horoboundary of a Hilbert geometry} \label{sec:4}

In this section, we describe the detour metric on parts of the horoboundary of
a Hilbert geometry using the characterisation of its Busemann points
obtained in \cite{Wa1}.
To present the results it is convenient to work with Hilbert's projective
metric. We begin by recalling some notions from \cite{Wa1}.
Given an open cone $C\subseteq\mathbb{R}^{n+1}$,
the \emph{open tangent cone at} $z\in\partial C$ is defined by
\[
\tau(C,z) =
   \{\lambda(x-z)\in\mathbb{R}^{n+1}\colon \mbox{$\lambda>0$ and $x\in C$}\}.
\]
Observe that $C=\tau(C,0)$.

\begin{lemma}
\label{lem:tangent_cone_formula}
For each $ z\in\partial C$ we have
\[\tau(C,z)=\{u\in\mathbb{R}^{n+1}\colon \langle \phi, u\rangle>0\mbox{ for all } \phi\in C^*\setminus\{0\} \mbox{ with }\langle \phi, z\rangle =0\}.\]
\end{lemma}
\begin{proof}
The inclusion $\subseteq$ is clear. To prove the opposite inclusion let
$Z=\{ \phi\in C^*\setminus\{0\} \colon \langle \phi, z\rangle =0\mbox{ and }\|\phi \|=1\}$. Suppose that $u\in\mathbb{R}^{n+1}$ is such that $\langle \phi, u\rangle > 0 $ for all $\phi \in Z$.  As $Z$ is compact, $\alpha =\min_{\phi \in Z} \langle \phi, u\rangle>0$. Let $0<\epsilon <\alpha/\|u\|$ and
$W_1=\{\psi \in C^*\setminus\{0\}\colon \|\psi\|=1 \mbox{ and }\|\psi -\phi\|<\epsilon 
\mbox{ for some }\phi\in Z\}$.
Then
\begin{eqnarray*}
\langle \psi, u\rangle & = &\langle \phi,u\rangle +\langle \psi,u\rangle-\langle \phi,u\rangle  \\
 & \geq & \langle \phi,u\rangle -\|\psi-\phi\|\|u\|\\
  & \geq & \alpha -\epsilon \|u\|>0, 
  \end{eqnarray*}
  where $\phi\in Z$ with $\|\psi-\phi\|<\epsilon$.
Now let
$W_2=\{\psi \in C^*\setminus\{0\}\colon \|\psi\|=1 \mbox{ and }\|\psi -\phi\|\geq \epsilon 
\mbox{ for all }\phi\in Z\}$. Denote
$\beta =\min_{\psi\in W_2}  \langle \psi,u\rangle$ and
$\gamma = \min_{\psi\in W_2} \langle \psi,z\rangle>0$.
Note that it suffices to show that $x=\mu u +z\in C$ for some $\mu>0$. Take
$0<\mu<|\gamma/\beta|$ and remark that if $\psi\in W_2$, then
\[
\langle \psi,\mu u\rangle + \langle \psi,z\rangle\geq \mu\beta +\gamma>0.
\]
We also have that
\[
\langle \psi,\mu u\rangle + \langle \psi,z\rangle>0
\]
for all $\psi\in W_1$. Thus, $x\in C$ and we are done.
\end{proof}

Given a collection $\Pi$ of open cones in $\mathbb{R}^{n+1}$, we write 
\[
\Gamma(\Pi) = \{\tau(T,z)\colon \mbox{$T\in\Pi$ and $z\in\partial T$}\}.
\]
Starting with $C$ and iterating this operation gives a collection
of open cones
\begin{equation*}
\mathcal{T}(C)=\bigcup_{k=1}^n \Gamma^k(\{C\}), 
\end{equation*}
where $\Gamma^{k+1}(\{C\})=\Gamma(\Gamma^k(\{C\})$ for all $k$. 
In particular, if $C\subseteq\mathbb{R}^{n+1}$ is an open polyhedral cone
with $N$ facets, then there exist $N$ facet defining functionals
$\psi_1,\ldots,\psi_N\in C^*$ such that 
\[
C=\{x\in\mathbb{R}^{n+1}\colon \psi_i(x)>0\mbox{ for }i=1,\ldots,N\}.
\]  
In this case it can be shown that
\[
\mathcal{T}(C)=\{C_I\colon
               \mbox{$I$ is a non-empty subset of $\{1,\ldots,N\}$}\}, 
\]
where $C_I=\{x\in\mathbb{R}^{n+1}\colon \psi_i(x)>0\mbox{ for all }i\in I\}$. 

It is instructive to determine the Busemann points that come from straight-line
geodesics in the Hilbert geometry. In fact, we will need this result later. 
\begin{lemma}\label{lem:4.1} 
If $C\subseteq\mathbb{R}^{n+1}$ is an open cone, and $\gamma(t)=(1-t)z+ty$,
with $t\in (0,1]$, is a straight-line geodesic connecting $z\in\partial C$
to $y\in C$, then 
\begin{align*}\label{eq:4.3}
\lim_{t\to 0} d_C(x,\gamma(t))-d_C(b,\gamma(t))
      & = \mathcal{RF}_C(x,z)-\mathcal{RF}_C(b,z)\\
      &\quad \quad +\mathcal{F}_{\tau(C,z)}(x,y)-\mathcal{F}_{\tau(C,z)}(b,y)
           \end{align*}
for each $x\in C$.
\end{lemma}
\begin{proof}
It follows from (\ref{eq:2.2}) that  
\begin{equation}\label{eq:4.4}
\begin{split}
\lim_{t\to 0} \mathcal{RF}_C(x,\gamma (t))-\mathcal{RF}_C(b,\gamma (t))
    & = \lim_{t\to 0}\quad \log \sup_{\phi\in C^*\setminus\{0\}}
         \frac{(1-t)\langle \phi,z\rangle + t\langle \phi,y\rangle}
                {\langle \phi,x\rangle}  \\
    &  \quad \quad 	- \log\sup_{\phi\in C^*\setminus\{0\}}
         \frac{(1-t)\langle \phi,z\rangle+ t\langle \phi,y\rangle}
                {\langle \phi,b\rangle} \\
    & =\mathcal{RF}_C(x,z)-\mathcal{RF}_C(b,z)\\
\end{split}
\end{equation}
for each $x\in C$. 

By \cite[Lemma 3.3]{Wa1} we also know that  
\[
\mathcal{F}_C(x,\gamma(t))-\mathcal{F}_{\tau(C,z)}(x,\gamma(t))\to 0
\mbox{\quad as $t\to 0$}, 
\] 
for all $x\in C$. It follows from~(\ref{eq:2.2})
and Lemma~\ref{lem:tangent_cone_formula} that 
\begin{equation*}\label{eq:4.5} 
\begin{split}
\mathcal{F}_{\tau(C,z)}(x,\gamma(t))
     & = \log \sup_{\phi\in C^*\setminus\{0\}, \langle \phi,z\rangle =0}
                \frac{\langle \phi,x\rangle}{\langle \phi, (1-t)z+ty\rangle}\\
     & = \log\frac{1}{t} + \log \sup_{\phi\in C^*\setminus\{0\}, \langle \phi,z\rangle =0}
                \frac{\langle \phi,x\rangle}{\langle \phi,y\rangle}\\
     & = \log\frac{1}{t} + \mathcal{F}_{\tau(C,z)}(x,y). \\
\end{split}
\end{equation*}
 Thus, 
\begin{equation*}
\lim_{t\to 0} \mathcal{F}_C(x,\gamma(t))-\mathcal{F}_C(b,\gamma(t))
       = \mathcal{F}_{\tau(C,z)}(x,y)-\mathcal{F}_{\tau(C,z)}(b,y).
\end{equation*}
Combining this with~(\ref{eq:4.4}) completes the proof.
\end{proof}

To describe all the Busemann points, not only the tangent cone is needed,
but all the cones in $\mathcal{T}(C)\setminus\{C\}$.
According to~\cite[Lemma 4.1]{Wa1}, a sequence $(x_k)_k\subseteq C$ is an
$\epsilon$-almost geodesic  with respect to Hilbert's projective metric on $C$
if and only if
it is an $\epsilon'$-almost-geodesic under both the Funk metric and the
reverse-Funk metric on $C$. For $T\in\mathcal{T}(C)$ and $y\in T$, let
$f_{T,y}\colon T\to \mathbb{R}$ be defined by 
\[
f_{T,y}(x)=\mathcal{F}_T(x,y)-\mathcal{F}_T(b,y),
\]
where $b\in C$ is the fixed base-point. Likewise, for $z\in\overline{C}$,
we define $r_{C,z}\colon C\to \mathbb{R}$ by
\[
r_{C,z}(x)=\mathcal{RF}_C(x,z)-\mathcal{RF}_C(b,z).
\]
Following~\cite{Wa1}, we say that a sequence $(x_k)_k\subseteq C$
\emph{converges to $f\colon C\to\mathbb{R}$ in the Funk sense on $C$}
if $(f_{C,x_k})_k$ converges pointwise to $f$ on $C$.
Similarly, a sequence $(x_k)_k\subseteq C$
\emph{converges to $r\colon C\to\mathbb{R}$ in the reverse-Funk sense}
if $(r_{C,x_k})_k$ converges pointwise to $r$ on~$C$.

Much like the Busemann points in the Hilbert geometry, we can consider
Busemann points in the Funk and in the reverse-Funk geometries on $C$,
which are defined as follows. A function $f\colon C\to\mathbb{R}$ is a
\emph{Busemann point in the Funk geometry on $C$} if there exists a Funk metric
$\epsilon$-almost-geodesic $(x_k)_k\subseteq C$  which converges to $f$ in the
Funk sense and $f$ is not of the form
$\mathcal{F}_C(\cdot,p)-\mathcal{F}_C(b,p)$ for $p\in C$.
Similarly, a function $r\colon C\to\mathbb{R}$ is a
\emph{Busemann point in the reverse-Funk geometry on $C$} if there exists a
reverse-Funk metric $\epsilon$-almost-geodesic $(x_k)_k\subseteq C$
which converges to $r$ in the reverse-Funk sense and $r$ is not of the form
$\mathcal{RF}_C(\cdot,p)-\mathcal{RF}_C(b,p)$ for $p\in C$.

The following proposition, proved in \cite[Proposition 2.5]{Wa1},
describes the Busemann points of the reverse-Funk geometry. 
\begin{proposition}\label{prop:4.2}
Let $C\subseteq\mathbb{R}^{n+1}$ be a proper open cone.
The set of Busemann points of the reverse-Funk geometry on $C$ is
\[
\mathcal{B}_\mathcal{RF}=\{r_{C,x}\colon x\in\partial C\setminus\{0\}\}.
\] 
Moreover, a sequence $(x_k)_k$ in a cross-section of $C$ converges in the
reverse-Funk sense to $r_{C,x}\in \mathcal{B}_\mathcal{RF}$ if and only if
it converges to a positive multiple of $x$ in the norm topology.
\end{proposition}
The Busemann points of the Funk geometry are more complicated as the following
result~\cite[Proposition 3.11]{Wa1} shows.
\begin{proposition}\label{prop:4.3} 
If $C\subseteq\mathbb{R}^{n+1}$ is a proper open cone, then the set of Busemann
points of the Funk geometry on $C$ is
\[
\mathcal{B}_\mathcal{F}
   = \{f_{T,p\mid C} \colon
          \mbox{$T\in\mathcal{T}(C)\setminus\{C\}$ and $p\in T$}\}.
\]
\end{proposition}
Each Busemann point in the Hilbert geometry is the sum of a Busemann point
in the Funk geometry and a Busemann point in the reverse-Funk geometry.
Indeed, the following characterisation was obtained in~\cite[Section 4]{Wa1}. 
\begin{theorem}\label{thm:4.4} 
If $C\subseteq\mathbb{R}^{n+1}$ is a proper open cone, then the set of Busemann
points of the Hilbert geometry on $C$ is
\[
\mathcal{B}=\{r_{C,x}+f_{T,p\mid C} \colon
         \mbox{$x\in\partial C\setminus\{0\}$, $T\in\mathcal{T}(\tau(C,x))$,
                 and $p\in T$}\}.
\]
Moreover, for each $r_{C,x}+f_{T,p\mid C}\in \mathcal{B}$ there exists an
almost-geodesic that converges in the norm topology to $x$
and in the Funk sense to $f_{T,p}$.
\end{theorem}
Thus, if $(x_k)$ is an almost-geodesic converging to
$g=r_{C,x}+f_{S,p\mid C}\in \mathcal{B}$, and 
$h=r_{C,y}+f_{T,q\mid C}\in\mathcal{B}$, then, by Lemma~\ref{lem:3.2},
\begin{equation*}
\begin{split}
H(g,h) & = \lim_{k\to\infty} d_C(b,x_k)+h(x_k)\\
		& = \lim_{k\to\infty}
                     \Big{(}\mathcal{RF}_C(b,x_k)+r_{C,y}(x_k)\Big{)}
		   +\
                    \Big{(}\mathcal{F}_C(b,x_k)+f_{T,q}(x_k)\Big{)}.
\end{split}
\end{equation*} 
We will consider the two parenthesised  expressions  separately. 
Recall that for $x\in\overline{C}$ the \emph{face of} $x$ is defined as the set containing  those points $y\in\overline{C}$ such that the straight-line through $x$ and $y$ contains a open line segment $I$ with $x\in I$ and $I\subseteq \overline{C}$.  
\begin{proposition}\label{prop:4.5} 
Let $C\subseteq\mathbb{R}^{n+1}$ be a proper open cone and
$x,y\in\partial C\setminus\{0\}$.
Let $(x_k)_k$ be an almost-geodesic with respect to the reverse-Funk metric
converging to $x$ in the norm topology.
If $y$ lies in the face $F$ of~$x$, then 
\begin{equation}\label{eq:4.9} 
\lim_{k\to\infty} \mathcal{RF}_C(b,x_k)+r_{C,y}(x_k)
= \mathcal{RF}_C(b,x)+\mathcal{RF}_F(x,y)-\mathcal{RF}_C(b,y).
\end{equation} 
The limit is $\infty$ otherwise.
\end{proposition}
\begin{proof}
Each almost-geodesic (under the reverse-Funk metric) in $C$ that converges in the 
norm topology to $x$ converges to $r_{C,x}$ in the reverse-Funk sense by~\cite[Proposition 2.5]{Wa1}.
Therefore, we can argue just as in the proof of Lemma \ref{lem:3.2},
replacing the metric $d$ by $\mathcal{RF}_C(\cdot,\cdot)$, to conclude that 
\[
\lim_{k\to\infty} \mathcal{RF}_C(b,x_k)+r_{C,y}(x_k)
\] 
is independent of the (reverse-Funk metric) almost-geodesic $(x_k)_k$ converging to $x$ in norm.

Let us consider $(z_k)_k$ with $z_k = \frac{1}{k}b +(1-\frac{1}{k})x$
for all $k\geq 1$. 
As every straight-line segment is a geodesic under the reverse-Funk metric,
$(z_k)_k$ is an almost-geodesic. 
Note that as $C$ is a proper cone, $C^*$ has non-empty interior. Therefore there exists $\psi\in C^*$ such that $\langle \psi,y\rangle =1$ and
$\langle \psi,z\rangle >0$ for all $z\in C\setminus\{0\}$.

Define 
\[
u=\frac{x}{\langle \psi,x\rangle}
\mbox{\quad and\quad } 
u_k=\frac{z_k}{\langle \psi,z_k\rangle}
\quad\mbox{for each $k\geq 1$}.
\]

Recall that
$\mathcal{RF}_C(\alpha v,\beta w)=\log(\beta/\alpha)+\mathcal{RF}_C(v,w)$
for all $\alpha,\beta>0$. Therefore 
\begin{eqnarray*}
\lim_{k\to\infty} r_{C,y}(z_k)
    & = & \lim_{k\to\infty}
             -\log \langle\psi,z_k\rangle +  \mathcal{RF}_C(u_k,y)
                      -\mathcal{RF}_C(b,y)\\
    & = & \lim_{k\to\infty}
             -\log \langle\psi,z_k\rangle +  \log \frac{|w_ky|}{|w_k u_k|}
                      -\mathcal{RF}_C(b,y),
\end{eqnarray*}
where $w_k$ is the point in the intersection of the straight line through $u_k$ and $y$ with $\partial C$ on the same side of $y$ as $u_k$.

Suppose that $y=\lambda x$ for some $\lambda>0$. So, $u=y$ and each $u_k$
lies on the straight-line segment connecting
$b'=b/\langle \psi,b\rangle$ and $y$. In this case, obviously, 
\[
\frac{|w_ky|}{|w_ku_k|}\to 1
\]
as $k$ tends to infinity.
Moreover, $\mathcal{RF}_C(b,x_k)$ converges to $\mathcal{RF}_C(b,x)$ and
$-\log \langle \psi,z_k\rangle$ converges to
$-\log \langle \psi,x\rangle =\log \lambda$ as $k$ tends to infinity. 
Since $\mathcal{RF}_C(b,x)=\mathcal{RF}_C(b,y)-\log \lambda$,
equality~(\ref{eq:4.9}) holds in this case. 

Now suppose that $y\in F$ and $y\neq \lambda x$ for all $\lambda>0$.
So, $y\neq u$ and $y$ in the face of $u$, since $u$ has the same face as $x$.
Therefore we can define $w$ to be the point in the intersection of
$\partial C$ with the straight line through $y$ and $u$ that is on the same
side of $y$ as $u$, and farthest away from $y$.
Since $y$ is in the face of $u$, we know that $w\neq u$. Thus, 
\begin{eqnarray*}
\lim_{k\to\infty} r_{C,y}(z_k)
    & = & \lim_{k\to\infty}
             -\log \langle\psi,z_k\rangle +  \mathcal{RF}_C(u_k,y)
                      -\mathcal{RF}_C(b,y)\\
    & = & -\log \langle\psi,x\rangle +  \log \frac{|wy|}{|wu|}-\mathcal{RF}_C(b,y)\\
    & = & \mathcal{RF}_F(x,y) -\mathcal{RF}_C(b,y).
\end{eqnarray*}
As $\mathcal{RF}_C(b,z_k)$ converges to $\mathcal{RF}_C(b,x)$ as $k$ tends to
infinity, equality~(\ref{eq:4.9}) holds. 

Finally, suppose $y$ is not in the face of $x$. So, $w=u$ and 
\[
\frac{|w_ky|}{|w_ku_k|}\to\infty\mbox{\quad as }k\to\infty.
\]
This completes the proof.
\end{proof}

Given an open cone $C\subseteq \mathbb{R}^{n+1}$ and a base-point $b\in C$,
we define for $x\in C$ a function $j_{C,x}\colon\mathbb{R}^{n+1}\to \mathbb{R}$
by 
\[
j_{C,x}(y)=\frac{M(y/x;C)}{M(b/x;C)}\mbox{\quad for }y\in\mathbb{R}^{n+1}.
\]
It follows from (\ref{eq:2.2}) that $j_{C,x}$ is convex.
Also note that $f_{C,x}(y) = \log j_{C,x}(y)$ for all $x,y\in C$.

We recall several concepts from convex analysis;
the reader may consult~\cite{Beer} for details.
The \emph{epi-graph} of a convex function
$f\colon\mathbb{R}^{n+1}\to\mathbb{R}$ is given by 
\[
\mathrm{epi}(f)
   =\{(x,\alpha)\in\mathbb{R}^{n+1}\times\mathbb{R} \colon f(x)\leq\alpha\}.
\]
The epi-graph is a convex set and can be used to define a topology on the
space $\Lambda(\mathbb{R}^{n+1})$ of proper, lower semi-continuous,
convex functions on $\mathbb{R}^{n+1}$ as follows.
A sequence $(f_k)_k$ in $\Lambda(\mathbb{R}^{n+1})$ is said to converge in the
\emph{epi-graph topology} to $f$ if the epi-graphs $\mathrm{epi}(f_k)$ converge
to $\mathrm{epi}(f)$ in the Painlev\'e-Kuratowski topology.
Here a sequence of closed sets $(A_k)_k$ in $\mathbb{R}^{n+1}\times \mathbb{R}$
converges to  $A$ in the \emph{Painlev\'e-Kuratowski topology} if
\[
\mathrm{Ls} A_k
   :=\bigcap_{k\geq 0}\Big{(}\overline{\bigcup_{i\geq k} A_i}\Big{)} 
\]
and 
\[
\mathrm{Li} A_k := \bigcap_{(k_i), k_i\to\infty}
            \Big{(}\overline{\bigcup_{i\geq 0}A_{k_i}}\Big{)}
\]
satisfy $A=\mathrm{Li} A_k=\mathrm{Ls} A_k$. 

We write $j^*_{C,x}\colon\mathbb{R}^{n+1}\to\mathbb{R}\cup\{\infty\}$
to denote the \emph{Legendre-Fenchel transform} of $j_{C,x}$, so 
\[
j^*_{C,x}(\phi)=\sup_{y\in\mathbb{R}^{n+1}}\langle \phi,y\rangle - j_{C,x}(y)
                      \mbox{\quad for $\phi\in\mathbb{R}^{n+1}$}.
\]
The Legendre-Fenchel transform is a homeomorphism on the space
$\Lambda(\mathbb{R}^{n+1})$ 
with respect to the epi-graph topology~\cite[Proposition~7.2.11]{Beer}.
Furthermore it was proved in \cite[Lemma 3.15]{Wa1} that if
$T\subseteq\mathbb{R}^{n+1}$ is an open cone, then for each $x\in T$
we have that 
\[
j^*_{T,x}(\phi) = \left\{\begin{array}{ll}
 0 & \mbox{if $\phi\in\{\psi\in T^*\colon M(b/x;T)
                  \langle \psi,x\rangle\leq 1\}$}\\
  \infty  &\mbox{otherwise}.
\end{array}\right.
\]
For $T\in\mathcal{T}(C)$ and $x\in T$, define
\[
U_{T,x}=\{ \psi\in T^*\colon M(b/x;T)\langle \psi,x\rangle > 1\}. 
\]  
\begin{proposition}\label{prop:4.6}
Let $f_{S,p\mid C}$ and $f_{T,q\mid C}$ be Busemann points
of the Funk geometry on a proper open cone $C\subseteq\mathbb{R}^{n+1}$,
with $S$ and $T$ in $\mathcal{T}(C)\backslash\{C\}$.
Then
\[
\liminf_{k\to\infty} \mathcal{F}_C(b,x_k)+ f_{T,q}(x_k)
   = \left\{
         \begin{array}{ll}
            \mathcal{F}_S(b,p)+\mathcal{F}_T(p,q)-\mathcal{F}_T(b,q),
                   &\mbox{if $S\subseteq T$,}\\
            \infty, & \mbox{otherwise,}
         \end{array}\right.
\]
where the infimum is taken over all sequences in $C$ converging to
$f_{S,p\mid C}$ in the Funk sense on $C$.
\end{proposition}
\begin{proof}
Let $(x_k)_k$ be any sequence in $C$ converging to $f_{S,p}$ in the Funk sense.
By Lemma~4.15 of~\cite{Wa1}, $j_{C,x_k}$ converges to $j_{S,p}$
in the epigraph topology, and so $j^*_{C,x_k}$ converges to $j^*_{S,p}$.
Let $y\in C^*$ be such that $j^*_{S,p}(y)= \infty$.
The properties of epi-convergence imply that $j^*_{C,x_k}(y)$ converges to
$\infty$. Therefore, $y\in U_{C,x_k}$ for $k$ large enough.

\newcommand\dotprod[2]{\langle{#1},{#2}\rangle}

Observe that
\begin{align}
 \mathcal{F}_C(b,x_k) + f_{T,q}(x_k)
   &= \log\Big(
      M(b/x_k;C)
      \sup_{z\in T^*} \frac{\dotprod{z}{x_k}}{\dotprod{z}{q}}
      \frac{1}{M(b/q;T)}
      \Big) \\
\label{eqn:logsup}
   &= \log\Big(
      \sup_{z\in T^*\cap U_{C,x_k}} \frac{1}{\dotprod{z}{q}}
      \frac{1}{M(b/q;T)}
      \Big).
\end{align}

Suppose that $S$ is not a subset of $T$. Then $T^*$ is not a
subset of $S^*$ and we can consider a point $y\in T^* \backslash S^*$.
As $j^*_{S,p}(\alpha y)= \infty$ for all $\alpha>0$, we know that 
$\alpha y\in U_{C,x_k}$ for $k$ large enough.
So,
\begin{align*}
\liminf_{k\to\infty} \mathcal{F}_C(b,x_k)+ f_{T,q}(x_k)
   \geq \log\frac{1}{\dotprod{\alpha y}{q}}\frac{1}{M(b/q;T)}.
\end{align*}
But $\alpha$ can be chosen to be as small as we like, and so, in this case,
\begin{align*}
\liminf_{k\to\infty} \mathcal{F}_C(b,x_k)+ f_{T,q}(x_k) = \infty.
\end{align*}

Now suppose that $S \subseteq T$. For any $y\in U_{S,p}$ we know that 
$j^*_{S,p}(y)= \infty$. Thus, as before, $y\in U_{C,x_k}$
for all $k$ large enough.
Therefore, from~(\ref{eqn:logsup}),
\begin{align*}
\liminf_{k\to\infty} \mathcal{F}_C(b,x_k)+ f_{T,q}(x_k)
   &\geq \log\Big(
      \sup_{z\in T^*\cap U_{S,p}} \frac{1}{\dotprod{z}{q}}
      \frac{1}{M(b/q;T)}
      \Big) \\
   & = \log\Big(
      M(b/p;S)
      \sup_{z\in T^*} \frac{\dotprod{z}{p}}{\dotprod{z}{q}}
      \frac{1}{M(b/q;T)}
      \Big) \\
   & = \mathcal{F}_S(b,p) + \mathcal{F}_T(p,q) - \mathcal{F}_T(b,q).
\end{align*}

We now wish to show that this bound can be attained by a judicious choice
of the sequence $(x_k)_k$. 

Since $S\in\mathcal{T}(C)\backslash\{C\}$, there exists a finite sequence
of cones
$(S_k)_{1\le k\le N}$ such that $S_k\in\Gamma(\{S_{k-1}\})$ for all $1<k\le N$,
and $S_1=C$ and $S_N=S$.
Let $x\in\partial S_{N-1}$ be such that $S_N = \tau(S_{N-1}, x)$.
Define the constant sequence $x_k = p$, for all $k\in\mathbb{N}$.
Obviously, $(x_k)$ converges to $f_{S,p}$ in the Funk sense on $S_N$.
Let $(w_k)_k$ be a sequence of points in $S_{N-1}$ such that
$W=\bigcup_k \{w_k\}$ is dense in $S_{N-1}$ and contains the basepoint~$b$.
For each $k\in\mathbb{N}$, let $y_k= (1-\lambda_k) x + \lambda_k x_k$,
where the sequence $(\lambda_k)_k$ of positive reals is chosen so that,
for each $k\in\mathbb{N}$,
\begin{align}
\label{eqn:busemann2}
y_k &\in S_{N-1},
\qquad\text{and}\\
\label{eqn:busemann3}
\Big|\mathcal{F}_{S_{N-1}}(w,y_k) - \mathcal{F}_{S_N}(w,y_k)\Big| &< \frac{1}{k},
   \qquad\text{for all $w\in\{w_0,\dots,w_k\}$}.
\end{align}
Inclusion~(\ref{eqn:busemann2}) holds when $\lambda_k$ is small enough,
and, by \cite[Lemma 3.3]{Wa1},
the same is true for~(\ref{eqn:busemann3}).
By \cite[Lemma 3.1]{Wa1},
\begin{align}
\label{eqn:busemann6}
\mathcal{F}_{S_N}(w,y_k)
   &= \mathcal{F}_{S_N}(w,x_k) - \log\lambda_k,
\qquad \text{for all $k\in\mathbb{N}$ and $w\in S_N$}.
\end{align}
Let $w\in W$. For $k\in\mathbb{N}$ large enough,
both $b$ and $w$ are in $\{w_0,\dots,w_k\}$.
So, applying~(\ref{eqn:busemann3}) and~(\ref{eqn:busemann6}) twice each, we get
\begin{equation*}
\Big|\mathcal{F}_{S_{N-1}}(w,y_k) - \mathcal{F}_{S_{N-1}}(b,y_k)
            - \mathcal{F}_{S_N}(w,x_k) + \mathcal{F}_{S_N}(b,x_k) \Big|
     < \frac{2}{k}.
\end{equation*}
We conclude that $\mathcal{F}_{S_{N-1}}(w,y_k) - \mathcal{F}_{S_{N-1}}(b,y_k)$ converges to $f_{S,p}(w)$ as $k$ tends to infinity.
Since this holds for all $w$ in a dense subset of $S_{N-1}$,
we see that $(y_k)$ converges to $f_{S,p}$ in the Funk sense on $S_{N-1}$.

Since $x\in[0]_S$ and $S\subseteq T$, we have that $x\in[0]_T$.
Therefore, by \cite[Lemma 3.1]{Wa1} again,
\begin{align}
\label{eqn:busemann7}
\mathcal{F}_T(y_k,q)
   &= \mathcal{F}_T(x_k,q) + \log\lambda_k,
\qquad \text{for all $k\in\mathbb{N}$}.
\end{align}
We combine~(\ref{eqn:busemann3}), (\ref{eqn:busemann6}),
and~(\ref{eqn:busemann7}) to get
\begin{align*}
\mathcal{F}_{S_{N-1}}(b,y_k) + \mathcal{F}_T(y_k,q)
        < \mathcal{F}_{S_N}(b,x_k)+ \mathcal{F}_T(x_k,q) + \frac{1}{k},
\end{align*}
for all $k\in\mathbb{N}$.

We can iterate the above argument to get a sequence $(z_k)$ in $C$ such that
$z_k$ converges to $f_{S,p\mid C}$ in the Funk sense on $C$, and such that
\begin{align*}
\mathcal{F}_C(b,z_k) + \mathcal{F}_T(z_k,q)
        < \mathcal{F}_S(b,p) + \mathcal{F}_T(p,q) + \frac{N}{k},
\end{align*}
for all $k\in\mathbb{N}$.
Taking the limit inferior and subtracting $\mathcal{F}_{T}(b,q)$, we get
\begin{align*}
\liminf_{k\to\infty} \mathcal{F}_C(b,z_k)+ f_{T,q}(z_k)
   & \leq \mathcal{F}_S(b,p) + \mathcal{F}_T(p,q) - \mathcal{F}_T(b,q).
\end{align*}
\end{proof}
Reasoning exactly as in the proof of Lemma \ref{lem:3.2} with 
$\mathcal{F}_C$ for $d$ gives the following result.
\begin{lemma}\label{lem:4.7.2} 
If $f_{S,p\mid C}$ and $f_{T,q\mid C}$ are Busemann points
of the Funk geometry on a proper open cone $C\subseteq\mathbb{R}^{n+1}$,
with $S$ and $T$ in $\mathcal{T}(C)\backslash\{C\}$, and $(z_k)_k$ is an
almost-geodesic in $C$ with respect to the Funk metric that
converges to $f_{S,p\mid C}$ in the Funk sense, then 
\[
\lim_{k\to\infty} \mathcal{F}_C(b,z_k)+f_{T,q}(z_k) = \inf_{(x_k)_k}\liminf_{k\to\infty} \mathcal{F}_C(b,x_k)+ f_{T,q}(x_k)
 \]
where the infimum is taken over all sequences in $C$ converging to $f_{S,p\mid C}$
in the Funk sense on $C$.
\end{lemma}
By combining Propositions~\ref{prop:4.5} and~\ref{prop:4.6}, we obtain the
following formula for the detour metric. 
\begin{theorem}\label{thm:4.7} 
If $g= r_{C,x} + f_{S,p\mid C}$ and $h=r_{C,y}+f_{T,q\mid C}$ are Busemann
points of the Hilbert geometry on a proper open cone
$C\subseteq\mathbb{R}^{n+1}$, then
\[
\delta(g,h)= \left\{ \begin{array}{ll} d_F(x,y)+d_S(p,q) &\mbox{if $x$ and $y$ have the same face $F$, and $S=T$,} \\
    \infty & \mbox{otherwise.}
\end{array}\right.
\]
\end{theorem}
\begin{proof}
Using~\cite[Lemma~4.3]{Wa1} and the formulae in the proof
of~\cite[Theorem~1.1, p.~524]{Wa1}, we get that there exists an almost-geodesic
$(x_k)_k$ in $C$ converging to $x$ in the norm topology and to $f_{S,p\mid C}$
in the Funk sense.
Recall that each almost-geodesic under Hilbert's projective metric 
is an almost-geodesic under the Funk metric and the reverse-Funk metric. 

Therefore we can combine Lemmas~\ref{lem:3.2} and~\ref{lem:4.7.2}
and Propositions~\ref{prop:4.5} and~\ref{prop:4.6} to deduce 
\begin{eqnarray*}
\delta(g,h) & = & H(g,h)+H(h,g)\\
                  & = & \mathcal{RF}_F(x,y)+\mathcal{RF}_F(y,x)
                          + \mathcal{F}_T(p,q) + \mathcal{F}_S(q,p)\\
                  & = & d_F(x,y)+d_S(p,q), 
\end{eqnarray*}
if $x$ and $y$ have the same face $F$ and $S=T$.
In the contrary case, we get that $\delta(g,h)=\infty$.
\end{proof}

\section{Isometric actions on parts} \label{sec:5}

We now analyse how isometries between polyhedral Hilbert geometries act
on parts. By Lemma~\ref{lem:3.4}, each isometry $g\colon X\to Y$ preserves
the detour metric, and hence maps parts to parts. If $X$ is a Hilbert geometry,
then it follows from Theorem~\ref{thm:4.7} that there is a one-to-one
correspondence between the parts of the horoboundary of $(X,d_X)$ and pairs of the form
$(F,U)$, where $F$ is a (relatively) open face of the open cone $C_X$ generated
by $X$, and $U\in\mathcal{T}(\tau(C_X,z))$ for some $z$ in $F$.
Moreover, the part corresponding to $(F,U)$ is isometric to
$(F\times U', d_{F\times U'})$, where $U'=U/[0]_U$ and  
\[
d_{F\times U'}((x,u),(y,v))= d_F(x,y)+d_{U'}(u,v)
    \mbox{\quad for all $x,y\in F$ and $u,v\in U'$}.
\]
A part of a polyhedral Hilbert geometry $(X,d_X)$ is called a
\emph{vertex part} if the corresponding pair is of the form
$(F_z,\tau(C_X,z))$, where $F_z$ is a ray through a vertex
$z\in\partial X\subseteq\partial C_X$ of $X$.
It is said to be a \emph{facet part} if the pair is of the form 
$(F,\tau(C_X,z))$, where $F$ is a (relatively) open facet of $C$, i.e., $\dim F =n$, and $z\in F$. 
Note that for a facet part, $\tau(C_X,z)$ is the open half-space
$\{x\in\mathbb{R}^{n+1}\colon \langle\phi,x\rangle>0\}$ with $\phi\in C_X^*$
the facet defining functional of $F$. The main objective of this section is to
prove that an isometry between polyhedral Hilbert geometries either maps vertex
parts to vertex parts, and facet parts to facet parts, or it interchanges them.
Recall that, as the topology of the Hilbert metric coincides with the
norm topology, isometric Hilbert geometries must have the same dimension.
We start with the following basic observation.

\begin{lemma}\label{lem:5.1} 
If $(X,d_X)$ and $(Y,d_Y)$ are polyhedral Hilbert geometries and
$g\colon X\to Y$ is an isometry, then $g$ maps parts corresponding to pairs
of the form $(F,\tau(C_X,z))$, with $F$ a relatively open face of the cone
$C_X$ generated by $X$ and $z\in F$, to parts corresponding to pairs
$(F',\tau(C_Y,z'))$, with $F'$ a relatively open face of the cone
$C_Y$ generated by $Y$ and $z'\in F'$.
\end{lemma}
\begin{proof}
Note that the dimension of the Hilbert geometry on an open cone
$T\subseteq\mathbb{R}^{n+1}$ is equal to $n-\mathrm{dim}\,[0]_T$.
Thus, for $z\in\partial X\subseteq\mathbb{R}^{n+1}$, the dimension of the
Hilbert geometry on $\tau(C,z)$ is greater than the dimension of any other
open cone in $\mathcal{T}(\tau(C,z))$. Clearly, if $F$ is a relatively open
face of $C$ and $z\in F$, then
$\mathrm{dim}\, [0]_{\tau(C,z)} = \mathrm{dim}\, F$.
On the other hand, the Hilbert geometry on $F$ has dimension equal to
$\mathrm{dim}\, F -1$. Thus, the parts corresponding to pairs
$(F,\tau(C,z))$, with $F$ a relatively open face of $C$ and $z\in F$,
are precisely those that have maximal dimension $n-1$. The same is true for
parts of $(Y,d_Y)$ corresponding to pairs $(F',\tau(C_Y,z'))$, with $F'$ a
relatively open face of the cone $C_Y$ generated by $Y$ and $z'\in F'$.
As the topology of the Hilbert geometry coincides with the norm topology,
it follows from Theorem \ref{thm:4.7} that $g\colon X\to Y$ must preserve
the dimension of the parts. This completes the  proof.  
\end{proof}

Before we start proving the main result of this section we recall,
for definiteness, several basic concepts from metric geometry and prove some
auxiliary statements. Given a metric space $(X,d)$ and an interval
$I\subseteq\mathbb{R}$, a map $\gamma\colon I\to X$ is called a \emph{geodesic}
if 
\[
d(\gamma(s),\gamma(t))=|s-t|\mbox{\quad  for all }s,t\in I. 
\] 
If $I=[a,b]$ with $-\infty<a<b<\infty$, the image of $\gamma$ is called a
\emph{geodesic segment} connecting $\gamma(a)$ and $\gamma(b)$.
Likewise if $I=\mathbb{R}$, we call the image of $\gamma$ a
\emph{geodesic line}. A geodesic line is said to be \emph{unique}
if for each finite interval $[s,t]\subset\mathbb{R}$, the geodesic segment
$\gamma([s,t])$ is the only one connecting $\gamma(s)$ and $\gamma(t)$.
A subset $U\subseteq X$ is said to be \emph{geodesically closed}
if for every $u,v\in U$, every geodesic segment connecting $u$ and $v$
is contained in $U$. 
In the Hilbert geometry, since straight-line segments are geodesic segments,
geodesically closed sets are convex.

The following result is well known.
\begin{lemma}[{\cite[Proposition 2]{dlH}}]\label{lem:5.2}
Let $(X,d_X)$ be a Hilbert geometry.
If $\ell$ is a straight-line intersecting $X$ and $\ell$ intersects
$\partial X$ at an extreme point, then $\ell\cap X$ is a unique-geodesic line.
Conversely, if $\Gamma$ is a unique-geodesic line in $(X,d_X)$,
then $\Gamma=\ell\cap X$ for some straight-line $\ell$. 
\end{lemma} 

The following elementary topological fact will be useful. 
\begin{lemma}\label{lem:5.3} 
Let $X\subseteq\mathbb{R}^n$ be an open bounded convex set.
If $U$  is a non-empty convex subset of $X$, and $U$ is closed in $X$ and
homeomorphic to $\mathbb{R}^m$, then $U$ is the intersection of $X$ with
an $m$-dimensional affine space. 
\end{lemma} 
\begin{proof}
Let $A=\mathrm{aff}\, U$. Clearly $A$ is $m$-dimensional.
Since $U$ is convex and homeomorphic to $\mathbb{R}^m$, it must be open in $A$.
Remark that $X\cap A$ is also open in $A$ and contains $U$.
Therefore $U$ is open in $X\cap A$. But by assumption $U$ is also closed in
$X\cap A$, and so $U=X\cap A$, since $U$ is non-empty and
connected.  
\end{proof} 

We say a Hilbert geometry $(X,d_X)$ is~\emph{trivial} if $X$ consists of
a single point. 
\begin{proposition}\label{prop:5.4} 
Let $(Y,d_Y)$ and $(Z,d_Z)$ be non-trivial Hilbert geometries and suppose that $Y\times Z$ is equipped with the metric, 
\[
d_{Y\times Z} ((y,z),(y',z'))= d_Y(y,y')+d_Z(z,z')
\mbox{\quad for $y,y'\in Y$ and $z,z'\in Z$.}\] 
Then $(Y\times Z,d_{Y\times Z})$ is not isometric to any  Hilbert geometry. 
\end{proposition}
\begin{proof}
Let $\ell_Y\subseteq Y$ be a geodesic line such that one of its end-points is an extreme point of $Y$. Likewise let $\ell_Z\subseteq Z$ be a geodesic line with one of its end-points an extreme point of $Z$. Note that by Lemma \ref{lem:5.2} both $\ell_Y$ and $\ell_Z$ 
are unique-geodesic lines. Obviously, $\ell_Y\times \ell_Z$ is homeomorphic to $\mathbb{R}^2$ and closed in $(Y\times Z,d_{Y\times Z})$. We now show that $\ell_Y\times \ell_Z$ is also geodesically closed. Let $(y,z)$ and $(y',z')$ be points in $\ell_Y\times \ell_Z$ and let $\Gamma$ be a geodesic segment in $Y\times Z$ connecting them. By definition of the metric $d_{Y\times Z}$, the projection $\Gamma_Y$ of $\Gamma$ to $Y$ is a geodesic segment connecting $y$ and $y'$ in $Y$.
As $\ell_Y$ is a unique-geodesic line, the only geodesic segment connecting
$y$ to $y'$ in $Y$ is the straight-line segment $[y,y']$.
Therefore, $\Gamma_Y\subseteq \ell_Y$. By the same argument
$\Gamma_Z\subseteq \ell_Z$.
We conclude that $\Gamma \subseteq \ell_Y\times \ell_Z$. 

For the sake of contradiction suppose that $h$ is an isometry mapping
$(Y\times Z,d_{Y\times Z})$ onto a Hilbert geometry $(X,d_X)$.
Then $U=h(\ell_Y\times \ell_Z)$ is homeomorphic to $\mathbb{R}^2$ and 
closed in $(X,d_X)$. Moreover, $U$ is geodesically closed and hence convex.
Thus, by Lemma \ref{lem:5.3}, $U$ is the intersection of $X$ with an
affine plane. This implies that it is itself a Hilbert geometry.
Note that $(\ell_Y\times \ell_Z,d_{Y\times Z})$ is isometric to
$\mathbb{R}^2$ with the $\ell_1$-norm, $\|(x_1,x_2)\|_1=|x_1|+|x_2|$
for $(x_1,x_2)\in\mathbb{R}^2$.
According to Foertsch and Karlsson \cite{FK}, the only Hilbert geometry
isometric to a 2-dimensional normed space is the  Hilbert geometry on a
$2$-simplex. In that case, however, the unit ball of the norm is hexagonal,
and hence it cannot be isometric to the $\ell_1$-norm on $\mathbb{R}^2$.
This is the desired contradiction. 
\end{proof}

\begin{corollary}\label{cor:5.5} 
If $(X,d_X)$ and $(Y,d_Y)$ are polyhedral Hilbert geometries and $g\colon X\to Y$ is an isometry, then $g$ maps the collection of vertex parts and facet parts of the horoboundary of $(X,d_X)$ to the 
collection of vertex parts and facet parts of the horoboundary of  $(Y,d_Y)$.
\end{corollary}
\begin{proof}
We may consider $X$ and $Y$ to be open subset of $\mathbb{R}^n$ for some
$n\geq 1$.
Let $P$ be a vertex part or facet part of the horoboundary of $(X,d_X)$.
According to Theorem \ref{thm:4.7}, $P$ is isometric to an $n-1$ dimensional
Hilbert geometry. Therefore the part $g(P)$ of the horoboundary of $(Y,d_Y)$ with the detour
metric must also be isometric to such a geometry. If $(F,U)$ is the pair
corresponding to the part $g(P)$, then by Lemma \ref{lem:5.1},
$F$ is a relatively open face of the cone $C_Y\subseteq \mathbb{R}^{n+1}$
generated by $Y$, and $U=\tau(C_Y,z)$ for some $z\in F$.
From Proposition~\ref{prop:5.4} and Theorem~\ref{thm:4.7},
it follows that either $F$ is the ray through a vertex of $Y$,
in which case $g(P)$ is a vertex part, or $F$ is a relatively open facet
of $C_Y$ and $\tau(C_Y,z)$ is a half-space, in which case $g(P)$ is a
facet part.      
\end{proof} 
We will now show that there are only two types of isometries between polyhedral
Hilbert geometries: namely, those that map vertex parts to vertex parts,
and facet parts to facet parts, and those that interchange them. 
\begin{theorem}\label{thm:5.6} 
If $(X,d_X)$ and $(Y,d_Y)$ are polyhedral Hilbert geometries and
$g\colon X\to Y$ is an isometry, then either $g$ maps vertex parts to
vertex parts, and facet parts to facet parts, or it interchanges them.
\end{theorem} 
\begin{proof}
By Corollary \ref{cor:5.5}, it suffices to prove that if a facet part of the horoboundary of
$(X,d_X)$ is mapped to a vertex part of the horoboundary of $(Y,d_Y)$ under $g$, then every facet part gets maps to a vertex part and every vertex part gets mapped to a facet part.
So, suppose that $g$ maps the facet part corresponding to $(F,\tau(C_X,z))$,
with $z\in F$, to the vertex part $(F_v,\tau(C_Y,v))$, where $F_v$ is the ray
through the vertex $v\in \overline Y$. 
Now let $F'$ be a facet adjacent to $F$. For the sake of contradiction,
suppose that the facet part corresponding to the pair $(F',\tau(C_X,z'))$, with $z'\in F'$, is not mapped to a vertex part of the horoboundary of $(Y,d_Y)$. By Corollary \ref{cor:5.5}, its image must be a facet part of $(Y,d_Y)$. Let us denote the corresponding pair of this facet part by $(g(F'),\tau(C_Y,v'))$. 

Note that the vertex $v$ is adjacent to $g(F')$, as otherwise there would be
a unique-geodesic line $\Gamma$ connecting $v$ to a point in $g(F')$.
This would imply, however, that $g^{-1}(\Gamma)$ is a unique-geodesic
line connecting points in the facets $F$ and $F'$, which is impossible
by~\cite[Proposition 2]{dlH}. 

Now let $\gamma_1\colon \mathbb{R}\to X$ be a unique-geodesic line such that  $\lim_{t\to\infty}\gamma_1(t)\in g(F')$. There exists a unique-geodesic line $\gamma_2\colon \mathbb{R}\to Y$ such that $\lim_{s\to\infty}\gamma_2(s)=v$ and $\gamma_2(0)=\gamma_1(0)$. Put $r=\gamma_1(0)$, and remark that $\mathrm{aff}\,(\gamma_1,\gamma_2)$ is 2-dimensional. 

Let
\begin{equation*}
(x\mid y)_r = \frac{1}{2}\Big(d(x,r)+d(y,r)-d(x,y)\Big)
\end{equation*}
denote the Gromov product with base-point $r$. 
For $i=1,2$, let $\gamma_i(\pm\infty)$ denote the limits as $s,t\to\pm\infty$,
respectively. In particular, $\gamma_2(\infty)=v$. For each $m>0$,
there exist $s_m$ and $t_m$ greater than $m$ such that the straight-line
through $\gamma_1(t_m)$ and $\gamma_2(s_m)$ is parallel to the straight-line
through $\gamma_1(\infty)$ and $v$. 
Note that there exists a constant $C_1$ such that the following inequality holds. 
\begin{equation}\label{eq:5.1}
\begin{split}
d_X(\gamma_2(s_m),\gamma_2(0)) & = \log \Big{(}\frac{|v\gamma_2(0)|}{|v\gamma_2(s_m)|}\frac{|\gamma_2(s_m)\gamma_2(-\infty)|}{|\gamma_2(0)\gamma_2(-\infty)|}
\Big{)}\\
   & \geq \log |v\gamma_2(0)|-\log|v\gamma_2(s_m)|\\
   &\geq C_1-\log|v\gamma_2(s_m)|
\end{split}
\end{equation}
for all $m>0$. There also exists a constant $C_2$ such that 
\begin{equation}\label{eq:5.2}
\begin{split}
d_X(\gamma_1(t_m),\gamma_2(s_m)) & = \log [u'_m,\gamma_1(t_m),\gamma_2(s_m),v'_m]\\
& = \log \frac{|u'_m\gamma_2(s_m)|}{|u'_m\gamma_1(t_m)|}
+ \log |\gamma_1(t_m)v'_m| - \log |\gamma_2(s_m)v'_m|\\
   &\leq C_2-\log|\gamma_2(s_m)v'_m|.
\end{split}
\end{equation}
for all $m>0$. 
Substituting (\ref{eq:5.1}) and (\ref{eq:5.2}) into the Gromov product gives 
\[
\limsup_{m\to\infty} 2(\gamma_1(t_m)\mid\gamma_2(s_m))_r\geq \limsup_{m\to\infty} 
d_X(\gamma_1(t_m),\gamma_1(0))+\log \frac{|\gamma_2(s_m)v'_m|}{|\gamma_2(s_m)v|}+C_3,
\]
for some constant $C_3$. By construction 
\[
\frac{|\gamma_2(s_m)v'_m|}{|\gamma_2(s_m)v|}
\]
is constant for large $m$. Since $d_X(\gamma_1(t_m),\gamma_1(0))\to\infty$ as $m$ tends to $\infty$, we find that 
\[
\limsup_{m\to\infty} 2(\gamma_1(t_m)\mid\gamma_2(s_m))_r = \infty.
\]

Note that  $g^{-1}$ is an isometry that maps $Y$ onto $X$ and 
\[
(g^{-1}(\gamma_1(t))\mid g^{-1}(\gamma_2(s)))_{g^{-1}(r)} =
(\gamma_1(t)\mid\gamma_2(s))_r 
\] 
for each $s$ and $t$. Thus, 
\begin{equation}\label{eq:5.3}
\limsup_{m\to\infty} 2(g^{-1}(\gamma_1(t_m)\mid g^{-1}(\gamma_2(s_m)))_{g^{-1}(r)} =\infty.
\end{equation}
As $g$ maps the facet part $(F', \tau(C_X,z'))$ to the facet part
$(g(F'),\tau(C_Y,v'))$ and the facet part $(F,\tau(C_X,z))$ to the
vertex part $(F_v,\tau(C_Y,v))$, it follows from Lemmas~\ref{lem:4.1}
and~\ref{lem:5.2} that $g^{-1}(\gamma_1(t_m))$ converges to $x\in F'$
and $g^{-1}(\gamma_2(s_m))$ converges to $y\in F$ as $m$ tends to $\infty$.
As the straight-line segment $[x,y]\not\subseteq \partial X$,
we deduce from \cite[Theorem 5.2]{KN} that 
\[
\limsup_{m\to\infty} 2(g^{-1}(\gamma_1(t_m)) \mid g^{-1}(\gamma_2(s_m)))_{g^{-1}(r)} <\infty, 
\]
which contradicts (\ref{eq:5.3}). 

We can reason in the same way from $F'$, and conclude that $g$ maps each facet part to a vertex part. It remains to show that $g$ maps vertex parts to facet parts. Again we argue by contradiction. So, let $P$ be a vertex part of $(X,d_X)$  corresponding to  
$(F_v,\tau(C_X,v))$, and suppose hat $g$ maps $P$ to a vertex part $(F'_u,\tau(C_Y,u))$ of $(Y,d_Y)$. There exists a unique-geodesic line $\Gamma_{p}\subseteq X$ connecting $v$ to a point  
$p\in F$, where $F$ is a facet of $C_X$ whose closure does not contain $v$. We already know that the facet part $(F,\tau(C_X,p))$ of $(X,d_X)$ is mapped to a vertex part $(F'_w,\tau(C_Y,w))$ of $(Y,d_Y)$. The image of $\Gamma_{p}$ under $g$ is  a unique-geodesic line, $\Gamma'_{p}$, which connects $u$ to $w$ in $(Y,d_Y)$ by Lemmas \ref{lem:4.1} and \ref{lem:5.2}. This implies that $u$ and $w$ do not lie in the same closed facet of $Y$, and hence $\Gamma'_{p}$ must be the straight-line segment $(u,w)$ in $Y$ for each $p\in F$, which contradicts the fact that  $g$ is 
one-to-one.  
\end{proof}

We shall prove that every isometry between polyhedral Hilbert geometries
that maps vertex parts to vertex parts, and hence facet parts to facet parts,
is a collineation. In addition, we shall see that isometries that interchange
vertex parts and facet parts only exist between two $n$-simplices with
$n\geq 2$.

\section{Isometries that map vertex parts to vertex parts} 

We first show that if an isometry between polyhedral Hilbert geometries maps
vertex parts to vertex parts, then it admits a continuous extension to the
norm boundary of its domain.
\begin{lemma}\label{lem:6.1} 
Let  $(X,d_X)$ and $(Y,d_Y)$ be polyhedral Hilbert geometries and 
$g\colon X\to Y$ be an isometry. If $g$ maps vertex parts to vertex parts, 
then $g$ extends continuously to $\partial X$. 
\end{lemma} 
\begin{proof} 
Let $n=\dim X=\dim Y$. For $m\leq n$, let $X_m$ be the union of the relative
open faces of $\overline{X}$ with dimension at least $m$.
In particular, $X_n=X$. We use an inductive argument with the following
hypothesis: the map $g$ extends continuously to $X_m$, and every straight-line
segment $(v,x)\subseteq X_m$ with $v$ a vertex of $X$ is mapped onto a
straight-line segment $(g(v),y)$ in $\overline{Y}$,
where $g(v)$ is the vertex of $Y$ corresponding to the part that is the image
under $g$ of the part of $v$.

To see that the assertion is true for $m=n$, remark that $(v,x)$ is (part of)
a unique-geodesic line, and hence $g((v,x))$ is a straight-line segment
$(w,y)$ with $w\in\partial Y$, by Lemma~\ref{lem:5.2}.
Let $(v_k)_k$ be a sequence in $(v,x)$ converging to $v$.
It follows from Lemma~\ref{lem:4.1} that $(v_k)_k$ converges in the
horofunction compactification to a Busemann point of the form
$r_{C_X,v} +f_{\tau(C_X,v),p}$ for some $p\in X$.
Thus, by assumption, $(g(v_k))_k\subseteq (w,y)$ must converge to
$r_{C_Y,g(v)}+f_{\tau(C_Y,g(v)),q}$ for some $q\in Y$.
As $g(v_k)$ converges to $w$, it follows that $g(v)=w$. 

Now suppose the assertion is true for some $m\in\{1,\dots,n\}$.
Let $F$ be a relative open face of $\overline{X}$ of dimension $m-1$.
Fix a vertex $v_F$ of $X$ not lying in $\overline F$.
For each $x\in F$, consider the straight-line segment $(v_F,x)$, which,
by our choice of $v_F$, is contained in $X_m$.
By the induction hypothesis, $g$ maps $(v_F,x)$ onto a straight-line segment
$(g(v_F),y)$. Define $g$ on $F$ by $g(x)=y$.

We claim that this extension of $g$ to $X_{m-1}$ is continuous.
Let $(x_k)_{k}$ be a sequence of points in $X_{m-1}$ converging to
some point $x$ of $F$. Without loss of generality we may assume that $(x_k)_k$
lies within $F\cup X_m$.
Any point $z\in (g(v_F), g(x))$ is the image under $g$ of a point $u\in(v_F,x)$.
Moreover, we can find a sequence $(u_k)_k$ in $X_m$ converging to $u$
such that $u_k \in (v_F, x_k)$ for all $k$.
By one part of the induction hypothesis, $g(u_k) \in (g(v_F), g(x_k))$ for
all $k$.
By the other part, $g(u_k)$ converges to $z= g(u)$, since $u$ is in $X_m$.
Therefore, every limit point $y'$ of $(g(x_k))_k$ satisfies $z\in(g(v_F), y']$.
By letting $z$ approach $g(x)$, we conclude that $y'=g(x)$,
and hence $g$ is continuous on $X_{m-1}$

To complete the induction step, let $(v,x)\subseteq X_{m-1}$ be a straight-line segment with $v$ a vertex of 
$X$. Suppose $s,t\in (v,x)$ and $s\in (v,t)$. Let  $(s_k)_k$ and $(t_k)_k$ be sequences in $X$ with $s_k\in (v,t_k)$ for all $k$, and such that $s_k\to s$ and $t_k\to t$ as $k\to\infty$. By the induction hypothesis, the straight-line segment $(v,t_k)$ is mapped onto 
$(g(v),g(t_k))$, so that $g(s_k)\in (g(v),g(t_k))$ for all $k$. As $g$ is continuous on $X_{m-1}$ we conclude that $g(s)\in [g(v),g(t)]$. Thus, the image of $(v,x)$ under $g$ is contained in a straight-line segment $(g(v),y)$ for some $y\in\overline{Y}$. Moreover, as $g$ is continuous, $g((v,x))$ must be connected, and hence it is a straight-line segment.  
\end{proof} 

We also need the following two lemmas. 
\begin{lemma} \label{lem:6.2} 
Let $U\subseteq\mathbb{R}^n$ be an $n$-dimensional compact convex set.
If $x_0,\ldots,x_n\in\partial U$ form an $n$-simplex, then for each $u\in U$,
there exists $x_m$ such that $\ell_{u,x_m}$ intersects
$\mathrm{aff}\,(\{x_0,\ldots,x_n\}\setminus\{x_m\})$ at a point in $U$.
\end{lemma}
\begin{proof}
Write $S=\mathrm{conv}\,(x_0,\ldots,x_n)$ to denote the $n$-simplex,
and for $m=1,\ldots,n$ define
$A_m=\mathrm{aff}\,(\{x_0,\ldots,x_n\}\setminus\{x_m\})$.
If $u\in S$, then $\ell_{u,x_m}$ intersects $A_m$ at a point in
$\mathrm{conv}(\{x_0,\ldots,x_n\}\setminus\{x_m\})\subseteq U$.
On the other hand, if $u\not\in S$, then for each $k$ we let $H_k$ be the
closed half-space containing $x_k$ with boundary $A_k$.
Obviously, $S$ is the intersection of these halfspaces,
and so there exists $m\in\{0,\dots,n\}$ such that $u$ is not in $H_m$.
Since $x_m$ is in $H_m$ any $u$ is not, the intersection of $\ell_{u,x_m}$
and $A_m$ lies in $[u,x_m]$. But $[u,x_m]$ is a subset of $U$ since $U$ is
convex.
\end{proof} 

\begin{lemma}\label{lem:6.3} 
Let $(X,d_X)$ and $(Y,d_Y)$ be polyhedral Hilbert geometries and
$g\colon X\to Y$ be an isometry having a continuous extension to $\partial X$.
If $x_0,\ldots,x_m\in\partial X$ are vertices of $X$ and
$u\in \mathrm{aff}\,(x_0,\ldots,x_m)\cap \overline{X}$,
then $g(u)\in\mathrm{aff}\,(g(x_0),\ldots,g(x_m))$.
\end{lemma}
\begin{proof}
We use induction on $m$.
The case $m=1$ is a direct consequence of Lemma~\ref{lem:5.2}.

Now suppose that the assertion holds for all $m<k$.
Let $x_0,\ldots,x_k$ be vertices of $X$.
By removing points we may assume that $\mathrm{conv}\,(x_0,\ldots,x_k)$
is an $k$-simplex. By Lemma~\ref{lem:6.2}, there exists $k^*$ such that
$\ell_{u,x_{k^*}}$ intersects
$\mathrm{aff}\,(\{x_0,\ldots,x_k\}\setminus\{x_{k^*}\})$ at some point
$z$ in $\overline{X}$. By the induction hypothesis
$g(z)$ is in $\mathrm{aff}\,(\{g(x_0),\ldots,g(x_k)\}\setminus\{g(x_{k^*})\}$.
Let $(v_i)_i$ be a sequence in $X$ converging to $u$.
Since $g$ extends continuously to the boundary and
$g(\ell_{v_i,x_{k^*}}\cap X) = \ell_{g(v_i),g(x_{k^*})}\cap Y$,
we find that $g(z)\in \ell_{g(u),g(x_{k^*})}$.
Thus, $g(u)$ is an affine combination of $g(z)$ and $g(x_{k^*})$,
and hence contained in
$\mathrm{aff}\,(g(x_0),\ldots,g(x_k))$.  
\end{proof} 

The next theorem shows that every isometry between polyhedral Hilbert
geometries mapping vertex parts to vertex parts is a collineation. 
\begin{theorem}\label{thm:6.4} 
Let $(X,d_X)$ and $(Y,d_Y)$ be polyhedral Hilbert geometries and
$g\colon X\to Y$ be an isometry. If $g$ and $g^{-1}$ extend continuously
to the boundary, then $g$ is a collineation.  
\end{theorem} 
\begin{proof}
We will use induction on $n=\dim X=\dim Y$.
Assume that $X$ is 1-dimensional. 
Let $a$ and $b$ be the points of $\partial X$, and let $x$ be any point in $X$.
Then $a$, $b$, $x$ form a projective basis for $\mathbb{P}^2$.
Hence there exists a unique collineation $h$ that coincides with $g$
on $a$, $b$ and $x$. Let $y\in X$ be between $x$ and $b$.
As $g$ extends continuously to $\partial X$, we must have that $g(y)$ is
between $g(x)$ and $g(b)$. Since $h$ preserves cross-ratios,
\[
[g(a),g(x),g(y),g(b)]=[a,x,y,b]=[g(a),g(x),h(y),g(b)].
\]
This equality uniquely determines $h(y)$ and hence $g$ and $h$ agree at $y$.
By interchanging the roles of $a$ and $b$ we conclude that $g$ and $h$ agree
on $X$.

Now assume that the assertion is true for all $k<n$.
Then we can find $n+1$ vertices $x_0, \dots,x_n$ of $\overline{X}$ that form
an $n$-simplex, which we denote by $S$.
Choose a point $y$ in the interior of $S$. The points $x_0, \dots, x_n, y$
form a projective basis for $\mathbb{P}^n$.

Note that, since $y$ is not in
$\mathrm{aff}\,(\{x_0,\ldots,x_n\}\setminus\{x_m\})$ for any $m$,
we can apply Lemma \ref{lem:6.3} to $g^{-1}$ and conclude that $g(y)$ is not
in the affine hull of $\{g(x_0),\ldots,g(x_n)\}\setminus\{g(x_m)\}$ for any $m$.
A similar argument shows that $g(x_i)$ is not in the affine hull of
$\{g(x_0), \dots,g(x_n)\}\setminus\{g(x_i)\}$ for any $i$.
It follows that $g(x_0), \dots,g(x_k),g(y)$ form a projective basis
for $\mathbb{P}^n$.
Therefore, there is a unique collineation $h$ agreeing with $g$ at 
$x_0, \dots,x_n$ and $y$.

For each $i\in\{0,\dots,n\}$, define
\[
L_i = \overline{X} \cap \ell_{x_i,y}
\mbox{\quad and\quad }
H_i = \overline{X} \cap \mathrm{aff}\,(\{x_0, \dots,x_n\}\setminus\{x_i\}).
\]
Since $y$ is in the interior of $S$, we have that $L_i$ intersects $H_i$
at a single point $z_i$.
Note that $g$ maps $L_i$ to $L'_i = \overline{Y}\cap\ell_{g(y),g(x_i)}$. 
For $i=1,\ldots,n$, let 
\[
H'_i = \overline{Y}
     \cap \mathrm{aff}\,(\{g(x_0), \dots,g(x_n)\}\setminus\{g(x_i)\}).
\]
By applying Lemma \ref{lem:6.3} to both $g$ and $g^{-1}$ we also know that
$g(H_i)=H_i'$ for all $i$. 
Therefore, $g(z_i)$ is the unique point of intersection of $L'_i$ and $H_i'$.
The collineation $h$ also maps $L_i$ to $L'_i$ and $H_i$ to $H'_i$,
and therefore $g(z_i)=h(z_i)$.

Let $X_i$ and $Y_i$ denote the relative interiors of $H_i$ and $H'_i$,
respectively. Equipped with the restrictions of $d_X$ and $d_Y$ respectively,
these sets become Hilbert geometries. Moreover, by Lemma \ref{lem:6.3},
the map $g$ restricted to $X_i$ is an isometry of $X_i$ onto $Y_i$.
Of course, $g_{|X_i}$ extends continuously to $\partial X_i$
and its inverse extends continuously to $\partial Y_i$.
So, we may apply the induction hypothesis to deduce that $g_{|X_i}$
is a collineation.
As $g$ and $h$ agree on $\{x_0, \dots,x_n,z_i\}\setminus\{x_i\}$,
which forms a projective basis for the projective closure of $X_i$,
we have that $g$ and $h$ agree on $H_i$, for each $i$. 

Let $p$ be in the interior of $S$. Define $p_0=\ell_{p,x_0}\cap H_0$ and
$p_1=\ell_{p,x_1}\cap H_1$. Since $g$ and $h$ agree on both $x_0$ and $p_0$,
they both map $\ell_{p_0,x_0}\cap\overline{X}$ to
$\ell_{g(p_0),g(x_0)}\cap\overline{Y}$.
Similarly, they both map $\ell_{p_1,x_1}\cap\overline{X}$
to $\ell_{g(p_1,g(x_1)}\cap\overline{Y}$.
We conclude that $g(p)=h(p)$, and hence $g$ and $h$ agree on the whole of $S$.

Let $\{u_0,\dots,u_n\}$ be a set of $n+1$ vertices of $X$ such that
$S'= \mathrm{conv}\,(u_0,\dots,u_n)$ is an $n$-simplex.
By the basis exchange property for affine spaces,
there exists an $i$ such that $u_i, x_1,\dots,x_n$ form a $n$-simplex.
Let $q$ be in the interior of $\mathrm{conv}\,(u_i, x_1,\dots,x_n)$.
The straight line  $\ell_{q,u_i}$ intersects the relative interior of
the facet $\mathrm{conv}\,(x_1,\ldots,x_n)$ of $S$. Therefore $\ell_{q,u_i}$
also intersects the interior of $S$.
Thus, $g$ and $h$ agree on at least three distinct points $u$, $v$,
and $w$ of $\ell_{q,u_i} \cap X$. Let $a$ be the point different from $u_i$
where $\ell_{q,u_i}$ intersects $\partial X$.
There exists a unique collineation $f$ that agrees with $g$ on
$a$, $u$, and $u_i$. The map $f$ is an isometry on $(a,u_i)$ and hence  
$f$ and $g$ agree on $\ell_{q,u_i} \cap X$.
Since $u$, $v$, $w$ forms a projective basis for the $1$-dimensional
projective space containing $\ell_{q,u_i}$, we find that $f$ and $h$ agree on
$\ell_{q,u_i} \cap X$, and hence $g$ and $h$ also agree on $\ell_{q,u_i}\cap X$.
Thus, we have shown that $g$ and $h$ are identical on the interior of
$\mathrm{conv}\,(u_i, x_1,\dots,x_n)$.
In fact, as $g$ has a continuous extension to $\partial X$, the maps $g$ and 
$h$ agree on $\mathrm{conv}\,(u_i, x_1,\dots,x_n)$.

Now note that we can iterate this procedure and replace, one-by-one,
the elements of $\{x_0,\dots,x_n\}$ with elements of $\{u_0,\dots,u_k\}$
to deduce that $g$ and $h$ are identical on $S'$.
By Carath\'{e}odory's theorem, every point in $\overline{X}$ can be written
as a convex combination of $n+1$ vertices of $X$.
Therefore $g$ and $h$ agree on the whole of  $\overline{X}$,
which shows that $g$ is a collineation.
\end{proof}

\section{Isometries that interchange vertex and facet parts}
\label{sec:7}

\begin{theorem}\label{thm:7.1}
Let $(X,d_X)$ and $(Y,d_Y)$ be polyhedral Hilbert geometries with
$\dim X=\dim Y\geq 2$. If there exists an isometry $g\colon X\to Y$
that maps vertex parts to facet parts, then $X$ and $Y$ are $n$-simplices. 
\end{theorem}
\begin{proof}
By Theorem \ref{thm:5.6}, we know that both $g$ and $g^{-1}$ map vertex parts
to facet parts and \textit{vice versa}.
Thus, it suffices to show that $X$ is an $n$-simplex.

To establish this we prove that its vertex set $V_X$ is affinely independent.
If $v\in V_X$, then there exists a relative open face $F_v$ of $\overline{X}$
such that $v$ is not in $\overline F_v$. Suppose that there exists another
vertex $u$ of $X$, different from $v$, that is not in $\overline F_v$.
Choose $p\in F_v$.
Let $\gamma\colon \mathbb{R}\to X$ and $\mu\colon\mathbb{R}\to X$ be
parametrisations of the unique-geodesic lines $(p,v)$ and $(p,u)$, respectively,
such that both $\gamma(t)$ and $\mu(t)$ converge to $p$ as $t$ tends
to $\infty$.

Note that, by Lemma \ref{lem:4.1}, both $\gamma$ and $\mu$ converge,
as $t$ tends to $\infty$, to the same Busemann point
$r_{C_X,p} +f_{\tau(C_X,p),q}$, where $q$ is any point in $X$.
Thus, $g\circ\gamma$ and $g\circ\mu$ converge, as $t$ tends to $\infty$, to the same Busemann point in $Y_B(\infty)$. By assumption this Busemann point is in a vertex part, and so is of the form $r_{C_Y,w}+f_{\tau(C_Y,w),s}$,
where $w$ is a vertex of $Y$ and $s\in Y$.

By Lemma \ref{lem:3.4}, the Busemann points in $X_B(\infty)$ corresponding
to $\gamma(t)$ and $\mu(t)$ with $t$ tending to $-\infty$ are mapped to
Busemann points in different facet parts of $(Y,d_Y)$.
Thus, $g((p,v))=(w,r)$ and $g((p,u))=(w,r')$ for some $r$ and $r'$ lying in
distinct facets of $\overline{Y}$.
However, by Lemma \ref{lem:4.1}, this implies that $g\circ\gamma$ and
$g\circ\mu$ converge, as $t$ tends to $\infty$, to different Busemann points
in the part of $w$, which is a contradiction.
\end{proof} 

\begin{proof}[Proof of Theorem \ref{thm:1.1}]
Suppose that $g$ is in $\mathrm{Isom}(X)$ and is not a collineation.
By Theorem~\ref{thm:6.4}, either $g$ or $g^{-1}$ does not extend continuously
to $\partial X$. From Theorem \ref{thm:5.6} and Lemma \ref{lem:6.1},
it follows that $g$ has to interchange vertex parts and facet parts
and $\dim X\geq 2$. It thus follows from Theorem \ref{thm:7.1} that $X$
is an $n$-simplex with $n\geq 2$ . 

The existence of an isometry that is not a collineation on any $n$-simplex
with $n\ge 2$, follows immediately from Theorem~\ref{thm:1.2},
which will be proved in the next section.
\end{proof}

\section{The isometry group of the simplex} \label{sec:8}
\begin{proof}[Proof of Theorem \ref{thm:1.2}]
It is known \cite{Nu1} that the $n$-simplex endowed with the Hilbert metric
is isometric to the normed vector space $V=\mathbb{R}^{n+1}/\sim$, where
$x\sim y $ if and only if $ x= y+ h(1,1,\ldots,1)$ for some $h\in\mathbb{R}$,
and norm 
\[
\| x\|_{\mathrm{var}}=
   \max_{i} x_i - \min_{j}x_j.
\]
We denote the equivalence class of $x\in\mathbb{R}^{n+1}$ by $[x]$. 
It is obvious that each element of $\mathbb{R}^n\rtimes \Gamma_{n+1}$, 
where $\Gamma_{n+1} =  \sigma_{n+1}\times\langle\rho\rangle$, 
is an isometry of $(V,\|\cdot\|_{\mathrm{var}})$.

By the Mazur-Ulam theorem, every isometry of $V$ is affine.
Let $g\colon V\to V$ be an isometry that fixes the origin.
Clearly the unit ball of $V$ is a polyhedron,
each vertex of which has exactly one representative in the set
\[
V_{\mathrm{var}} = \big\{ (b_0, \dots, b_n) \mid \text{$b_i \in \{0,1\}$ for all $i$} \big\}
\backslash \big\{ (0,\dots,0), (1,\dots,1) \big\}.
\]
This is the set of vertices of a hypercube with two diagonally opposite
corners removed. We see that there are $2^{n+1}-2$ vertices.

Edges of $B_\mathrm{var}$ are segments connecting vertices having
representatives in $V$ that differ on exactly one coordinate.
Thus, there are $n+1$ edges incident to every vertex, except for those whose
representative has exactly one coordinate equal to $0$ or $1$.
Let $V_0$ be the set of vertices whose representative has exactly one
coordinate equal to $0$, and let $V_1$ be the set of vertices whose
representative has one coordinate equal to $1$.
Since $(0,\ldots,0)$ and $(1,\ldots,1)$ are not in $V_\mathrm{var}$,
each vertex in $V_0\cup V_1$ is incident to exactly $n$ edges.
Since $g$ is linear, it preserves the number of edges incident to each vertex,
and so we conclude that $g$ leaves $V_0 \cup V_1$ invariant.

Now consider a subset $U$ of $V_0\cup V_1$ containing $n+1$ elements and having
the following properties: 
no element $U$ is the negative of another element in $U$,
and $\sum_{[u]\in U}[u]=[0]$.
It straightforward to verify that $U$ is equal to either $V_0$ or $V_1$.
Since the properties of $U$ are invariant under linear transformations,
$g$ maps $V_1$ either onto itself, or onto $V_0$.
As $V_1$ spans $V$, any linear map on $V$ is completely determined by its
values on $V_1$. Thus, if $g$ maps $V_1$ onto itself, then $g$ is a
permutation in $\sigma_{n+1}$. On the other hand, if $g$ maps $V_1$ onto $V_0$,
then $g$ is the composition of a permutation in $\sigma_{n+1}$ and $\rho$, as $V_0=-V_1$. We conclude that 
\[
\mathrm{Isom}(X)\cong \mathbb{R}^n\rtimes \Gamma_{n+1}.
\]

To determine the collineation group, let $C_X\subseteq \mathbb{R}^{n+1}$ be the
open cone generated by an $n$-simplex $X$ inside a hyperplane not containing
the origin. Any element $A$ of $\mathrm{GL}(n+1,\mathbb{R})$ that maps $C_X$
onto itself, maps the extreme rays of $C_X$ to extreme rays.
As the $n+1$ vertices of $X$ span $\mathbb{R}^{n+1}$, the map $A$ is completely
determined by its values on the vertices of $X$.
Thus, $A$ can be uniquely represented by a product of an $(n+1)\times (n+1)$
permutation matrix and an $(n+1)\times (n+1)$ positive diagonal matrix.
From this we conclude that
$\mathrm{Coll}(X)\cong \mathbb{R}^n\rtimes \sigma_{n+1}$. 
\end{proof}

We can go from the normed space representation of simplical Hilbert geometries
given above to the cone setting of Section~\ref{sec:2} by exponentiating
coordinate-wise. Indeed, let $\Phi$ be given by
$\Phi(x_1,\dots,x_{n+1})=(e^{x_1},\dots, e^{x_{n+1}})$.
Then, $\Phi$ is an isometry between $(V,||\cdot||_{\mathrm{var}})$ and
$(P_{n+1},d_{P_{n+1}})$, where $P_{n+1}$ is the interior of the standard positive cone. The map $\rho$ on $V$ corresponds to the map
$\rho' = \Phi\circ\rho\circ\Phi^{-1}$ on $P_{n+1}$, which takes the
coordinate-wise reciprocal. It is clear that $\rho'$ is both
order-reversing and homogeneous of degree $-1$.

Maps with these two properties exist on all
\emph{symmetric cones}, of which the cone $P_{n+1}$ is an example.
Indeed, recall \cite{FaK} that a proper open cone $C$ in a finite dimensional
real vector space $V$ with inner-product $\langle\cdot,\cdot\rangle$ is called
\emph{symmetric} if $\{A\in\mathrm{GL}(V)\colon A(C)=C\}$
acts transitively on $C$ and $C=C^*$, where 
\[
C^*=\{y\in V^*\colon \langle x,y\rangle>0\mbox{ for all }x\in\overline{C}\}
\]
is the (open) dual of $C$. 
The \emph{characteristic function} $\phi$ on $C$ given by, 
\[
\phi(x) =\int_{C^*} e^{-\langle x,y\rangle}dy\mbox{\quad for }x\in C,
\]
is homogeneous of degree $-\dim V$, so that Vinberg's \emph{$*$-map},
$x\in C\mapsto x^*\in C^*$, where $x^*=-\nabla\log\phi(x)$ for $x\in C$,
is homogeneous of degree $-1$.
The $*$-map is order-reversing on symmetric cones;
see~\cite[Proposition 3.2]{Kai}.
As a matter of fact, it was proved in \cite{Kai} that this property of the
$*$-map characterises the symmetric cones among the homogeneous cones.
The reader can verify that the map $\rho'$ above is the
$*$-map for the positive cone.

Since the $*$-map is order-reversing and homogeneous of degree $-1$,
it is non-expansive in Hilbert's projective metric on $C$; see~\cite{Nu1}.
But $(x^*)^* =x$ for all $x\in C$, so the $*$-map is actually an isometry under
this metric. Composing it with the canonical projection yields an isometry
of the Hilbert geometry on a section $X$ of $C$.
This isometry is not a collineation except when the symmetric cone
$C$ is a \emph{Lorentz cone},
\begin{align*}
\Lambda_n=\{(x_1,\ldots,x_n)\in \mathbb{R}^n
   \colon \mbox{$x_1>0$ and $x_1^2-x_2^2-\ldots -x_n^2>0$}\},
\end{align*}
for some $n\geq 2$.
To our knowledge there exist no other cones for which
$\mathrm{Isom}\,(X)$ differs from $\mathrm{Coll}\,(X)$.
In fact, we conjecture that $\mathrm{Isom}\,(X)$ and $\mathrm{Coll}\,(X)$
differ if and only if the cone generated by $X$ is symmetric and not Lorentzian,
in which case we believe the isometry group is generated by the
collineations and the isometry coming from the $*$-map.
This is known to be true for the cone of positive-definite Hermitian matrices;
see \cite{Mol}.

\end{document}